\documentclass[10pt, reqno]{article}
\usepackage{amsmath,amssymb,amsthm,graphicx}
\usepackage[left=2.5cm,right=2.2cm,top=2.5cm,bottom=3.5cm]{geometry}

\usepackage[breaklinks]{hyperref}
\hypersetup{
	colorlinks = true, 
	urlcolor = blue, 
	linkcolor = blue, 
	citecolor = blue 
}
\usepackage{orcidlink}
\usepackage{dsfont}
\usepackage{parskip}
\makeatletter
\def\thm@space@setup{%
  \thm@preskip=\parskip \thm@postskip=0pt
}
\makeatother
\newcommand\blankfootnote[1]{%
  \begingroup
  \renewcommand\thefootnote{}\footnote{#1}%
  \addtocounter{footnote}{-1}%
  \endgroup
}

\let\le\leqslant
\let\ge\geqslant
\let\a\alpha

\let\eps\varepsilon

\let\mc\mathcal
\let\mb\mathbb
\newcommand{\ZZ}{\mathbb{Z}}
\newcommand{\RR}{\mathbb{R}}
\newcommand{\NN}{\mathbb{N}}

\newcommand{\QQ}{\mathbb{Q}}
\newcommand{\CC}{\mathbb{C}}

\newcommand{\FF}{\mathbb{F}}
\newcommand{\EE}{\mathbb{E}}

\usepackage[noabbrev,capitalize,nameinlink]{cleveref}

\newtheorem{lemma}{Lemma}[section]
\newtheorem{proposition}[lemma]{Proposition}
\newtheorem{theorem}[lemma]{Theorem}

\newtheorem{fact}[lemma]{Fact}
\crefname{fact}{Fact}{Facts}

\theoremstyle{definition}
\newtheorem{definition}[lemma]{Definition}
\newtheorem{example}[lemma]{Example}
\newtheorem{remark}[lemma]{Remark}
\newtheorem*{remark*}{Remark}

\DeclareMathOperator{\Span}{span}
\DeclareMathOperator{\codim}{codim}

\newcommand{\prob}[1]{\mathbb{P}\left[#1\right]}
\newcommand{\probs}[2]{\mathbb{P}_{#1}\left[#2\right]}

\let\originalleft\left
\let\originalright\right
\renewcommand{\left}{\mathopen{}\mathclose\bgroup\originalleft}
\renewcommand{\right}{\aftergroup\egroup\originalright}

\crefformat{equation}{#2(#1)#3}

\title{Geometric Littlewood--Offord problems via lattice point counting}
\author{
Alexandr Grebennikov\footnote{Institute of Science and Technology Austria (ISTA), \href{mailto:aleksandr.grebennikov@ist.ac.at}{\nolinkurl{aleksandr.grebennikov@ist.ac.at}}.}
\and 
Matthew Kwan\footnote{Institute of Science and Technology Austria (ISTA), 
\href{mailto:matthew.kwan@ist.ac.at}{\nolinkurl{matthew.kwan@ist.ac.at}}.}
}

\begin{document}

    \maketitle

    \begin{abstract}
Consider nonzero vectors $a_{1},\dots,a_{n}\in\mathbb{C}^{k}$, independent
Rademacher random variables $\xi_{1},\dots,\xi_{n}$, and a set $S\subseteq\mathbb{C}^{k}$.
What upper bounds can we prove on the probability that the random
sum $\xi_{1}a_{1}+\dots+\xi_{n}a_{n}$ lies in $S$? We develop a
general framework that allows us to reduce problems of this type
to counting lattice points in $S$. We apply this framework with known results from diophantine geometry to prove various bounds when $S$ is a set of points in convex position, an algebraic variety,
or a semialgebraic set. In particular, this resolves conjectures of
Fox--Kwan--Spink and Kwan--Sauermann.

We also obtain some corollaries for the \emph{polynomial Littlewood--Offord
problem}, for polynomials that have bounded \emph{Chow rank} (i.e.,
can be written as a polynomial of a bounded number of linear forms).
For example, one of our results confirms a conjecture of Nguyen and Vu in the special case of polynomials with bounded Chow rank: if a bounded-degree polynomial
$F\in\mathbb{C}[x_{1},\dots,x_{n}]$ has bounded Chow rank and ``robustly
depends on at least $b$ of its variables'', then $\mathbb{P}[F(\xi_{1},\dots,\xi_{n})=0]\le O(1/\sqrt{b})$. We also prove significantly stronger bounds when $F$ is ``robustly
irreducible'', towards a conjecture of Costello.
    \end{abstract}

    \blankfootnote{Both authors are supported by ERC Starting Grant ``RANDSTRUCT'' No.~101076777.}

    \section{Introduction}
    Throughout this paper $\xi_1, \ldots, \xi_n$ will always denote a sequence of independent Rademacher random variables (that is, taking values $1$ or $-1$ with probability $1/2$).
    
    In 1943, motivated by their study of random algebraic equations, Littlewood and Offord \cite{littlewood-offord-43} considered the following question: given a sequence of $n$ nonzero real numbers $c_1, \ldots, c_n$, what is the maximum probability that the random variable $\xi_1 c_1 + \ldots + \xi_n c_n$ equals a given value\footnote{Here and for the rest of this introduction, we specialise to the ``discrete'' form of the Littlewood--Offord problem. (There is also a ``small-ball'' form, where one assumes that the $c_i$ have absolute value at least 1, and studies the likelihood that $\xi_1 c_1 + \ldots + \xi_n c_n$ falls in a given interval of radius 1.)}? They proved an upper bound of the form $O(\log n / \sqrt{n})$, which was sharpened by an elegant argument of Erd\H{o}s \cite{erdos-45} to the following precise result (now called the Erd\H{o}s--Littlewood--Offord theorem):
    \begin{equation}
    \sup_{z \in \RR}\prob{\xi_1 c_1 + \ldots + \xi_n c_n = z} \le 2^{-n}\binom{n}{\lfloor n/2 \rfloor} =\left(\sqrt{2/\pi} + o(1)\right)\frac{1}{\sqrt{n}}= O\left(\frac1{\sqrt n}\right).\label{eq:LO}
    \end{equation}
    Since then, the Erd\H{o}s--Littlewood--Offord theorem has been generalised in many different directions, and these results have found applications in a wide variety of  different fields (e.g.,\ random matrix theory, the theory of Boolean functions, extremal combinatorics; see \cite{nguyen-survey-13} for a survey). In this paper we introduce a general approach to attack several different questions in Littlewood--Offord theory of a ``geometric'' flavour.

    \subsection{Polynomial Littlewood--Offord problem}

    First, a natural direction of generalisation is to replace the linear form $\xi_1 c_1 + \ldots + \xi_n c_n$ by a polynomial $F(\xi_1, \ldots, \xi_n)$ of higher degree. This direction was first considered by Rosi\'nski and Samorodnitsky \cite{RS-96}, Costello, Tao and Vu \cite{costello-tao-vu-06}, and Razborov and Viola \cite{razborov-viola-13} in the contexts of L\'evy chaos, discrete random matrices and Boolean functions, respectively. 

It is widely believed that if $F$ is an $n$-variable degree-$d$
polynomial that is ``robustly nonzero'' then a bound
analogous to \cref{eq:LO} should hold. For example, it was conjectured by Nguyen and Vu (see \cite{meka-nguyen-vu-16, razborov-viola-13}) that if $F$ has at least $bn^{d-1}$ nonzero coefficients, then\footnote{Subscripts on asymptotic notation indicate quantities that should be treated as constants.}
\begin{equation}
\prob{F(\xi_{1},\ldots,\xi_{n})=0}\le O_{d}\left(\frac1{\sqrt{b}}\right).\label{eq:poly-LO}
\end{equation}
This is known to hold for $d=1$ (thanks to the Erd\H os--Littlewood--Offord
theorem) and for $d=2$ (thanks to recent work of Kwan and Sauermann~\cite{kwan-sauermann-23}).
For general $d$, the best available bound (due to Meka, Nguyen and
Vu~\cite{meka-nguyen-vu-16}, via a result of Kane~\cite{kane-14}) falls short 
of \cref{eq:poly-LO} by a factor of $(\log b)^{O_d(1)}$.
    
    If true, the bound in \cref{eq:poly-LO} is best-possible: for example,
    one can see this by considering the polynomial $(x_1+\dots+x_n)^d$. However, it is natural to wonder whether one can prove much stronger bounds if one makes assumptions to rule out this kind of example. Indeed, it was conjectured by Costello \cite{costello-13} that \cref{eq:poly-LO} can be significantly improved when the polynomial $F$ is ``robustly irreducible''. Though his original conjecture was recently disproved by Kwan, Sah and Sawhney (see \cite[Appendix B]{JKSW}), it is plausible that the following ``repaired'' version of his conjecture
still holds: consider any polynomial $F\in\CC[t_{1},\ldots,t_{n}]$
of degree $d\ge2$, and suppose that for any reducible\footnote{In this paper, we use the convention that the zero polynomial is reducible.} polynomial
$G$ of degree at least $d$, the difference $F-G$ has at least $bn^{d-1}$ nonzero  coefficients. 
Then, \[\prob{F(\xi_1,\dots,\xi_n)=0}\le O_{d,\eps}(b^{-1+\eps}).\]
We remark that a slight variation on this
``repaired'' conjecture was also suggested by Jin, Kwan, Sauermann and Wang~\cite{JKSW}; they observed that this bound, if true, would be best possible.

As our first results in this paper, we essentially resolve the above
conjectures, under an assumption that $F$ has ``bounded complexity''.
Formally, for a polynomial $F\in\FF[t_{1},\ldots,t_{n}]$ of degree
$d$ (over some field $\FF\subseteq\CC$), define its \textit{Chow
rank (over $\FF$)} to be the minimal number $c$ such that $F$ can
be represented as $\sum_{i=1}^{c}P_{i}$, where each $P_{i}$ is a
product of $d$ (not necessarily homogeneous) linear forms with coefficients
in $\FF$. One can check that the Chow rank of any homogeneous polynomial
of a fixed degree $d$ is ``equivalent'' to the tensor rank of its coefficient
tensor (in the sense that each of them is bounded by some function
of the other).
    
    \begin{theorem} \label{polynomials-all}
        Let $1 \le b \le n$ and $d, c \ge 1$ be integers. Let $F \in \CC[t_1, \ldots, t_n]$ be a polynomial of degree $d$ and Chow rank at most $c$.
        \begin{enumerate}
            \item[(1)] 
            Suppose that $F$ ``robustly depends on at least $b$ of its variables'', in the sense that it is not possible to make $F(\xi_1,\dots,\xi_n)$ identically zero by fixing fewer than $b$ of the $\xi_i$ to $\pm1$ values. (In particular, this holds whenever $F$ has at least $bn^{d-1}$ nonzero degree-$d$ coefficients\footnote{Note that the total number of terms in $F$ with degree less than $d$ is at most $d n^{d-1}$, so these terms can be essentially ignored.}.) 
            Then
            \[
            \prob{F(\xi_1, \ldots, \xi_n) = 0} \le O_{d, c}(b^{-1/2}).
            \] 
            \item[(2)] Suppose that $d \ge 2$ and $d\ne 3$. Also, suppose that for any reducible polynomial $G \in \CC[x_1, \ldots, x_n]$ of degree at most $d$, the difference $F-G$ has at least $bn^{d-1}$ nonzero coefficients.
            Then for any $\eps > 0$,
            \[
            \prob{F(\xi_1, \ldots, \xi_n) = 0} = O_{d, c, \eps}(b^{-1+\eps}).
            \]
        \end{enumerate}
    \end{theorem}

We will discuss our proof approach properly later in this introduction, but very briefly: we prove \cref{polynomials-all} via a connection to diophantine geometry.
Specifically, we leverage various known estimates for the number of lattice
points on algebraic varieties; the reason for the ``$d\ne 3$'' exception in \cref{polynomials-all}(2) is that uniform estimates for Heath-Brown's so-called \emph{dimension growth conjecture} are available for affine algebraic varieties of all degrees except 3.

It is worth remarking that related geometric considerations have already
played an important role in the resolution of the quadratic Littlewood--Offord
problem (i.e., the proof of \cref{eq:poly-LO} in the case $d=2$). Indeed, the first
bound of the form $b^{-1/2+o(1)}$, due to Costello~\cite{costello-13}, was proved via
a low-rank/high-rank dichotomy, using geometric techniques (related
to the Szemer\'edi--Trotter theorem) in the low-rank case, and using
completely different ``decoupling'' techniques in the high-rank
case. For Kwan and Sauermann's recent $O(1/\sqrt{b})$ bound~\cite{kwan-sauermann-23}, they
also took a geometric point of view for low-rank quadratic polynomials
(though their proof did not as cleanly split into a low rank and high
rank case). In general, it seems to us (in some very vague sense) that the worst-case behaviour for polynomial anticoncentration is driven by geometric considerations for ``low-complexity polynomials'', and driven by statistical/mixing considerations for ``high-complexity polynomials''.

Regarding the assumptions on $F$ in \cref{polynomials-all}: in the earliest work on the polynomial Littlewood--Offord problem, the usual assumption was that $F$ has many nonzero coefficients (as in the Nguyen--Vu conjecture at the start of the section). It was first observed by Razborov and Viola~\cite{razborov-viola-13}\footnote{The Razborov--Viola assumption only took multilinear degree-$d$ terms into account; the specific assumption in \cref{polynomials-all}(1) was first considered by Kwan and Sauermann~\cite{kwan-sauermann-23}.} that one can state polynomial Littlewood--Offord bounds with a much weaker assumption that $F$ ``robustly depends on many variables'', as in \cref{polynomials-all}(1). One can attempt to restate (the repaired version of) Costello's conjecture with an analogous type of assumption (namely, that it is not possible to make $F$ reducible by fixing fewer than $b$ variables), but such an assumption leads to a different worst-case bound. Indeed, by considering the polynomial $(x_1+\dots+x_{n/2})^d-(x_{n/2+1}+\dots+x_n)$, it is not hard to see that in this setting we cannot hope for a bound stronger than about $b^{-1+1/(2d)}$. We are able to match this lower bound up to logarithmic factors, as follows.

    \begin{theorem} \label{polynomials-1/d}
        Let $2 \le b \le n$ and $d, c \ge 1$ be integers, and let $\FF$ be a subfield of $\CC$. Let $F \in \FF[t_1, \ldots, t_n]$ be a degree-$d$ polynomial which is irreducible (over $\FF$) and has Chow rank at most $c$ (over $\FF$). Suppose that $F$ remains an irreducible degree-$d$ polynomial (over $\FF$) after any substitution of $\pm 1$ values into fewer than $b$ of its variables. Then
        \[
        \prob{F(\xi_1, \ldots, \xi_n) = 0} \le O_{d, c}(b^{-1 + 1/(2d)} (\log b)^{C_{d, c}}),
        \]        
        for some constant $C_{d, c}$ depending only on $d$ and $c$.
    \end{theorem}
    We remark that \cref{polynomials-1/d} also illustrates that our methods are applicable for polynomials that are reducible over $\CC$ but irreducible over a smaller field $\FF \subseteq \CC$ (provided that certain necessary lattice point enumeration estimates are available).
    
    \subsection{Littlewood--Offord problem for algebraic varieties}
    \label{sec:intro-varieties}
    It turns out that the polynomial Littlewood--Offord problem, under a bounded Chow rank assumption, can be naturally interpreted as a geometric problem in low-dimensional space. Indeed, recall that a polynomial $F \in \FF[t_1, \ldots, t_n]$ of degree $d$ and Chow rank $c$ can be represented as $P_1+\dots+P_c$, where each $P_i$ is a product of $d$ (not necessarily homogeneous)   linear forms. Alternatively, one can write this as
    \[
    F(t_1, \ldots, t_n) = f(L_1(t_1, \ldots, t_n), \ldots,  L_k(t_1, \ldots, t_n)).
    \]
    for $k = dc$, some homogeneous linear forms $L_1, \ldots, L_k$ with coefficients in $\FF$, and a polynomial $f \in \FF[x_1, \ldots, x_k]$. (In fact, the polynomial $f$ obtained this way has a certain specific form, but this turns out not to be useful for us.)
    Now, if we write $a_{ij}\in \mb F$ for the coefficient of $t_j$ in $L_i(t_1,\dots,t_n)$, and let $a_j = (a_{1j}, \ldots, a_{kj}) \in \FF^k$, then 
 the event that $F(\xi_1,\dots,\xi_n)=0$ can be interpreted as the event that $\xi_1 a_1 + \ldots + \xi_n a_n$ lies in the algebraic variety $S=\{x \in \CC^k: f(x) = 0\}\subseteq \mb C^k$. This observation was implicitly leveraged in \cite{costello-13,kwan-sauermann-23}.

In connection with their work on the polynomial Littlewood--Offord problem, Kwan and Sauermann made a general conjecture along these lines \cite[Conjecture 12.1]{kwan-sauermann-23}. Specifically, they conjectured that if one can form at least $b$ disjoint bases of $\mb C^k$ from the vectors $a_1, \ldots, a_n\in \mb C^k$ (this is a measure of how ``robustly'' $a_1,\dots,a_n$ span the space $\mb C^k$), then for any affine algebraic variety $S\in \mb C^k$ of dimension $\ell$ and degree $d$, we have
\[\prob{\xi_1 a_1 + \ldots + \xi_n a_n\in S}\le O_{d, k}(b^{-(k-\ell)/2}).\]
We remark that the $d=1$ case of this conjecture is essentially equivalent to a classical theorem of Hal\'asz~\cite{halasz-77}, and in the course of their resolution of the quadratic Littlewood--Offord problem (for quadratic polynomials of not necessarily bounded Chow rank), Kwan and Sauermann proved this conjecture for quadrics inside affine-linear subspaces \cite[Theorem 4.2]{kwan-sauermann-23}. We prove this conjecture in full generality.

    \begin{theorem} \label{varieties-(k-ell)/2}      
        Let $0 \le \ell \le k$ and $d \ge 1$ be integers. Let $S \subseteq \CC^k$ be a (possibly reducible) affine algebraic variety of dimension at most $\ell$ and degree at most $d$. Consider vectors $a_1, \ldots, a_n \in \CC^k$, and assume that one can form $b$ disjoint bases from them. Then 
        \[
        \prob{\xi_1 a_1 + \ldots + \xi_n a_n \in S} \le O_{d, k}\left(b^{-(k-\ell)/2}\right).
        \]
    \end{theorem}

    In fact, \cref{varieties-(k-ell)/2} is deduced as a corollary of the following more refined estimate.

    \begin{theorem} \label{varieties-1/d}
        Let $0 \le \ell \le k$ and $d \ge 2$ be integers. Let $S \subseteq \CC^k$ be an irreducible affine algebraic variety of dimension $\ell$ and degree $d$. Consider vectors $a_1, \ldots, a_n \in \CC^k$, and assume that one can form $b \ge 2$ disjoint bases from them. Then 
        \[
        \prob{\xi_1 a_1 + \ldots + \xi_n a_n \in S} \le O_{d, k}\left(b^{-(k-\ell+1-\frac{1}{d})/2} (\log b)^{C_{d, k}}\right), 
        \]
        for some constant $C_{d, k}$ depending only on $d$ and $k$.
    \end{theorem}
    We remark that although these results are stated over the field of complex numbers, one can deduce similar results for its smaller subfields 
    (this is necessary for the proof of \cref{polynomials-1/d}). Indeed, for a \textit{real} algebraic set $S_{\RR} \subseteq \RR^k$, one can consider its ``complexification'' $S_{\CC} \subseteq \CC^k$ (that is, the minimal complex variety containing $S_{\RR}$). It satisfies $S_{\CC} \cap \RR^k = S_{\RR}$, and the (complex) dimension of $S_{\CC}$ coincides with the (real) dimension of $S_{\RR}$.
    
    The main ingredient in the proof of \cref{varieties-1/d} is a general theorem (\cref{varieties-general}) estimating probabilities of the form $\prob{\xi_1a_1+\dots+\xi_na_n\in S}$ in terms of a certain ``lattice point density'' of $S$. \cref{varieties-1/d} in particular is proved using an estimate of Pila~\cite{pila-95,pila-96}, but in general one can plug other number-theoretic results about counting lattice points on varieties into \cref{varieties-general} to obtain analogous results (all the theorems mentioned so far are proved in this way). 
    The following example illustrates why the connection to number theory is not surprising in this context. 
    
    \begin{example} \label{example-lattice}
    Consider the standard basis vectors $e_1, \ldots, e_k$ of $\CC^k$, and consider the sequence of vectors $a_1, \ldots, a_{2mk}$ consisting of $2m$ copies of $e_i/2$ for each $i\in \{1,\dots,k\}$ (where $m$ is sufficiently large with respect to $k$). Then each coordinate of the random variable $X := \xi_1 a_1 + \ldots + \xi_{2mk} a_{2mk}$ is equal to $t \in \ZZ$ with probability ${2m \choose m + t} / 2^{2m}$, independently. Thus, by standard estimates on binomial coefficients, $X$ is essentially equidistributed (up to a multiplicative constant factor depending on $k$) over the integer points of the box $[-\sqrt{m}, \sqrt{m}]^k$. Therefore, the probability that $X$ lies in a variety $S \subseteq \CC^k$ is closely related to the proportion of integer points in this box that lie on $S$.
    \end{example}

    We emphasise that the assumption in \cref{varieties-1/d} only guarantees that the vectors $a_1, \ldots, a_n$ ``robustly span $\CC^k$'' as a vector space. In particular, they may be very far from resembling the standard generators of the integer grid. The main goal of this paper is to establish a connection between Littlewood--Offord-type questions and counting lattice points on varieties in this general setting.
    
    \subsection{Littlewood--Offord problem for general sets}

    In the above subsection, we have been discussing probabilities of the form $\prob{\xi_1a_1+\dots+\xi_na_n \in S}$, where $S$ is an algebraic variety. It is natural to wonder whether one can obtain similar upper bounds with more general (or completely different) assumptions on $S$: what are the geometric properties of a set $S$ which ensure that random sums are unlikely to fall in them?

    This general direction was recently initiated by Fox, Kwan and Spink \cite{fox-kwan-spink-23}, who investigated several very general conditions on $S$: namely, the condition that $S$ is a set of points in convex position (i.e., no point in $S$ can be represented as a convex combination of the others), and the condition that $S$ is ``definable with respect to an o-minimal structure'' (this is a very general model-theoretic notion which ensures that $S$ does not have ``infinitely oscillating'' structure).

    First, we discuss the case when $S$ is a set of points in convex position (which includes, in particular, boundaries of strictly convex bodies). In this setting, Fox, Kwan and Spink proved that for any nonzero vectors $a_1, \ldots, a_n$ in $\RR^k$ the probability that $\xi_1 a_1 + \ldots + \xi_n a_n$ lies in $S$ is at most $O_k(n^{-k/2^k})$ \cite[Theorem 1.9(1)]{fox-kwan-spink-23}, and conjectured that the stronger bound $n^{-1/2 + o_k(1)}$ should hold \cite[Conjecture 10.1]{fox-kwan-spink-23}. We show that this is indeed the case, and provide an asymptotically sharp bound.
    
    \begin{theorem} \label{convex-1/2}
        Let $S \subseteq \RR^k$ be a set of points in convex position. Consider arbitrary nonzero vectors $a_1, \ldots, a_n \in \RR^k$. Then, as $k$ is fixed and $n$ tends to infinity,
        \[
        \prob{\xi_1 a_1 + \ldots + \xi_n a_n \in S} \le \left(2\sqrt{2/\pi} + o_k(1)\right)n^{-1/2}.
        \]
    \end{theorem}

    The following example shows that this bound is asymptotically sharp: let $a \in \RR^k$ be an arbitrary nonzero vector, and take $S = \{-a, a\}$, $a_1 = a_2 = \ldots = a_{2n+1} = a$. However, this example is essentially one-dimensional.
    We obtain a stronger bound under the assumption that the vectors $a_1, \ldots, a_n$ ``robustly span $\RR^k$'' for $k \ge 2$ (and use it to deduce \cref{convex-1/2}).

    \begin{theorem} \label{convex-position}
        Let $S \subseteq \RR^k$ be a set of points in convex position. Consider vectors $a_1, \ldots, a_n \in \RR^k$, and assume that one can form $b \ge 2$ disjoint bases from them. Then 
        \[
        \prob{\xi_1 a_1 + \ldots + \xi_n a_n \in S} \le O_{k}\big(b^{-1+1/(k+1)} (\log b)^{C_k}\big)
        \]
        for some constant $C_k$ depending only on $k$.
    \end{theorem}

    In the current work, we do not pursue the most general situation when $S \subseteq \RR^k$ is a set ``definable with respect to an o-minimal structure''\footnote{It does seem to be possible to adapt our methods to this setting by combining them with the tools from \cite{fox-kwan-spink-23}, but this would require us to introduce and explain various concepts from o-minimal geometry, which are outside the scope of the present paper.}. Instead, we highlight the special case of semialgebraic sets: sets defined by a collection of polynomial equations and inequalities. When $S$ is a semialgebraic set which does not contain a line segment, Fox, Kwan and Spink proved that $\prob{\xi_1 a_1 + \ldots + \xi_n a_n \in S}\le n^{-1/2} (\log n)^{O_k(1)}$ \cite[Theorem 1.5]{fox-kwan-spink-23}. We observe that our approach allows us to remove the logarithmic factor.

    \begin{theorem} \label{semialgebraic}
        Let $S \subseteq \RR^k$ be a semialgebraic set, which does not contain a line segment. Consider arbitrary nonzero vectors $a_1, \ldots, a_n \in \RR^k$. Then
        \[
        \prob{\xi_1 a_1 + \ldots + \xi_n a_n \in S} \le O_S(n^{-1/2}).
        \]
    \end{theorem}

    This bound is sharp up to a multiplicative constant factor, but can be further improved under the assumption that the vectors $a_1, \ldots, a_n$ ``robustly span a high-dimensional subspace'', see \cref{semialgebraic-improved}.
    
    \subsection{Organization of the paper}

    This paper is organized as follows. In \cref{sec:outline} we provide a high-level outline of the proofs of \cref{convex-position,varieties-1/d}. In \cref{sec:notation-preliminaries} we introduce notation that will be used throughout the paper, and prove several preliminary lemmas. In \cref{sec:lattice-point-counting}, we prove a convenient intermediate result (\cref{lattice-points-reduction}) that reduces the case of ``polynomial point probabilities'' to lattice point counting. Next, in \cref{sec:convex}, we present the proofs of the results for sets of points in convex position (\cref{convex-position,convex-1/2}).

    For the rest of the paper (\cref{sec:algebraic-preliminaries,,sec:decoupling,,sec:main-proofs}) we turn to the algebraic setting. In \cref{sec:algebraic-preliminaries} we review several useful ingredients from algebraic geometry and number theory. \cref{sec:decoupling} contains our most technical theorem (\cref{structure-randomness}), which provides a decomposition of the ambient vector space into subspaces via an iterative decoupling procedure. Finally, in \cref{sec:main-proofs} we prove our general result for the algebraic setting (\cref{varieties-general}), and use it to deduce \cref{polynomials-all,,polynomials-1/d,,varieties-1/d,,varieties-(k-ell)/2,,semialgebraic}.

    \textit{Basic notation.} For a positive integer $n$, we write $[n]$ to denote the set $\{1, \ldots, n\}$.
    Our use of asymptotic notation is standard: for functions $f=f(n)$ and $g=g(n)$, we write $f=O(g)$ to mean that there is a constant $C$ such that $|f|\le C|g|$, $f=\Omega(g)$ to mean that there is a constant $c>0$ such that $f(n)\ge c|g(n)|$ for sufficiently large $n$, $f=\Theta(g)$ to mean that $f=O(g)$ and $f=\Omega(g)$, and $f=o(g)$ to mean that $f/g\to0$ as $n\to\infty$. Subscripts on asymptotic notation indicate quantities that should be treated as constants.
    All logarithms are assumed to be in base $2$.

    \textbf{Acknowledgements.} We would like to thank Tim Browning and Matteo Verzobio for helpful comments and references from the number theory literature.

    \section{Proof outline}\label{sec:outline}

    In this section we provide a high-level sketch of the proofs of \cref{convex-position} for sets of points in convex position and \cref{varieties-1/d} for algebraic varieties (all our other main results are either deduced from one of these two theorems, or from the intermediate lemmas in their proofs).
    
    The proofs of \cref{convex-position} and \cref{varieties-1/d} are based on similar ideas, though the proof of \cref{convex-position} is simpler.

    \subsection{Sets of points in convex position}

    To prove \cref{convex-position} we need to obtain an upper bound on the probability that $X := \xi_1 a_1 + \ldots + \xi_n a_n$ lies in our set $S$ of points in convex position. 
    We separately treat the cases when 
    \[
    \rho := \sup_{x \in \RR^k} (X = x) < n^{-C}
    \] 
    and when $\rho \ge n^{-C}$ (for some appropriately chosen $C$ depending on the dimension $k$). 
    
    \textbf{The ``spread-out'' case ($\rho < n^{-C}$). } In this case, we just apply a result of Fox, Kwan and Spink \cite[Theorem 1.9(2)]{fox-kwan-spink-23} which implies that
    \[
    \prob{X \in S} \le O_k(\rho^{1/(k2^{k-1})}).
    \]
    That is to say, the probability of lying in $S$ can be bounded in terms of the maximum \emph{point} concentration probability. This directly implies the conclusion of \cref{convex-position}, if $C$ is sufficiently large. The proof of the above bound (in \cite{fox-kwan-spink-23}) is based on a reduction to a K\H ov\'ari--S\'os--Tur\'an-type theorem in an auxiliary hypergraph, and some simple combinatorial consequences of the fact that the points in $S$ lie in convex position.
    
    \textbf{The ``concentrated'' case ($\rho \ge n^{-C}$). } In this second case, we take advantage of the \emph{inverse theory} for the linear Littlewood--Offord problem. Roughly speaking, the philosophy of this theory is that the only way for $\rho$ to be large is for the coefficients $a_1,\dots,a_n$ to have strong additive structure.
    
    Specifically, our main tool will be the ``optimal inverse theorem'' for the linear Littlewood--Offord problem, proved by Nguyen and Vu (\cite[Theorem 2.5]{nguyen-vu-11}, stated below as \cref{inverse}). It says that if $\rho\ge n^{-C}$, then almost all of the vectors $a_1, \ldots, a_n$ are contained in a common \emph{generalized arithmetic progression} (``GAP'' for short; see \cref{def:GAP}), whose rank is bounded in terms of $C$, and whose volume depends in an ``optimal'' way on $\rho$.

    In \cref{iterated-inverse}, we show how to \emph{iterate} this optimal inverse theorem, to prove that in fact the random variable $X = \xi_1 a_1 + \ldots + \xi_n a_n$ is \emph{approximately equidistributed} in a certain GAP of bounded rank. In other words, if we eliminate\footnote{To ``eliminate'' a vector $a_i$ just means to fix an outcome of the corresponding random variable $\xi_i$, and work in the resulting conditional probability space.} a few ``exceptional'' vectors $a_i$, then up to an affine-linear transformation, we can think of $X$ as being approximately the uniform distribution on the integer points in a box of the form $[-B_1,B_1]\times\dots\times[-B_q,B_q]$.

    As a result, the problem of upper-bounding $\prob{X \in S}$ reduces to the problem of upper-bounding the number of integer points in a box $[-B_1,B_1]\times\dots\times[-B_q,B_q]$ which lie in a certain affine-linear transformation of $S$. For this, we can take advantage of classical estimates in discrete geometry (in particular, we use an estimate due to Andrews~\cite{andrews-63}, stated in this paper as \cref{Andrews}).

    The full details of the proof of \cref{convex-position} appear in \cref{sec:convex}.

    \subsection{Algebraic varieties}
    To prove \cref{varieties-1/d} we need to obtain an upper bound on the probability that $X := \xi_1 a_1 + \ldots + \xi_n a_n$ lies in our algebraic variety $S$.
    
    If we try to proceed via the same dichotomy as for the proof of \cref{convex-position}, the ``concentrated'' case ($\rho \ge n^{-C}$) works in exactly the same way: the only change is that Andrews' theorem should be replaced by a result of Pila \cite{pila-95,pila-96} (stated in this paper as \cref{Pila}), counting integer points on an affine algebraic variety.
    
    Unfortunately, we encounter some issues in the ``spread-out'' case ($\rho < n^{-C}$). Recall that for \cref{convex-position} we used a result of Fox, Kwan and Spink bounding $\prob{X\in S}$ in terms of $\rho$. Fox, Kwan and Spink also proved a similar result that can be applied to algebraic varieties (\cite[Theorem 1.14]{fox-kwan-spink-23}), but it requires the additional assumption that the variety $S$ does not contain any affine lines. In general, without such an assumption on $S$ there is no nontrivial bound\footnote{For example, suppose that $S$ is the line $\{(x,y):x=0\}\subseteq \RR^2$, and suppose $a_1, \ldots, a_n\in \RR^d$ are defined by $a_i = (1, 2^i)$. Then $\xi_1 a_1 + \ldots + \xi_n a_n$ lies in $S$ with probability $\Theta(n^{-1/2})$, while $\rho$ is exponentially small.} on $\prob{X\in S}$ in terms of $\rho$.
    
    Therefore, we take a different point of view. Instead of separately considering two extreme cases, our argument can be seen as an \emph{interpolation} between these two cases. Namely, in \cref{structure-randomness} we obtain a decomposition of the ambient vector space $\CC^k$ into a direct sum $U \oplus W$ of a ``disordered'' subspace $U$ and a ``structured'' subspace $W$ (where the decomposition is chosen with respect to the sequence of vectors $a_1, \ldots, a_n$ and the variety $S$).

    Roughly speaking, the property we will guarantee for our ``structured'' subspace $W$ is that, after eliminating a few ``exceptional'' vectors $a_i$, the \emph{projection of $X$ onto $W$} concentrates on some point with polynomially large probability (at least $n^{-C}$ for some constant $C$). Let $\pi_W: \CC^k = U \oplus W \to W$ be the projection map, and let $S'$ be the maximal subset of $W$ such that we have $\pi_W^{-1}(S') \subseteq S$. Then $X$ lies in $\pi_W^{-1}(S')$ if and only if its projection $\pi_W(X)$ lies in $S'$, so $\prob{X \in \pi_W^{-1}(S')}$ can be estimated using the same approach as for the ``concentrated'' case described in the previous subsection (replacing Andrews' theorem with Pila's theorem, as described at the beginning of this subsection).
    
    By construction of $S'$, knowing the value of the projection $\pi_W(X)$ \emph{cannot} allow us to conclude that $X$ lies in $S \setminus \pi_W^{-1}(S')$: this always depends on the ``disordered'' coordinate of $X$ (corresponding to the subspace $U$) as well. 
    So, the property we will guarantee for our ``disordered'' subspace $U$ is simply that $X$ is very unlikely to lie in $S \setminus \pi_W^{-1}(S')$. Together with the above considerations, this gives the desired upper bound on the probability that $X$ lies in $S$.

    The proof of \cref{structure-randomness} is by an iterative procedure: we begin with $U = \CC^k$, $W = \{0\}$, and then repeatedly enlarge $W$ while keeping it ``structured'' (shrinking $U$ correspondingly). At each step of this procedure, we use a \emph{decoupling} argument (\cref{iterated-decoupling}), which relates the probability that $X = \xi_1 a_1 + \ldots + \xi_n a_n$ lies in a variety $S$ with the probability that it lies in certain linear subspaces. We refer the reader to the discussion in \cref{sec:decoupling} for more details.

    \section{Notation and preliminaries}\label{sec:notation-preliminaries}
    
    Let $A = (a_1, \ldots, a_n)$ be a sequence of vectors in a finite-dimensional vector space $V$ (over some subfield $\FF$ of $\CC$). We note that the order of the vectors $a_1, \ldots, a_n$ is irrelevant for us in this work, and the word ``sequence'' is used as a synonym for the word ``multiset''.

    For a subset $I \subseteq [n]$, $I = \{i_1, \ldots, i_m\}$ we define the subsequence $A[I] = (a_{i_1}, \ldots, a_{i_m})$. We also say that $A'$ is a subsequence of $A$ of size $m$ if there exists $I \subseteq [n]$, $|I| = m$ such that $A' = A[I]$. 

    We define the \textit{basis packing number} of a sequence $A$ to be the maximum number of disjoint bases of $V$ one can form from the vectors of $A$. Formally, the basis packing number of $A$ is the largest integer $b$ for which there exist $b$ pairwise disjoint subsets $I_1, \ldots, I_b \subseteq [n]$ such that for each $1 \le j \le b$ the subsequence $A[I_j]$ is a basis of $V$.

\begin{definition}\label{def:rho}
        Define the maximum point probability $\rho(A)$ by
    \[
    \rho(A) = \sup_{x \in V} \prob{\xi_1 a_1 + \ldots + \xi_n a_n = x}.
    \]
    More generally, for any set $S \subseteq V$ we define the maximum $S$-translate probability $\rho(A,S)$ by
    \[
    \rho(A, S) = \sup_{x \in V} \prob{\xi_1 a_1 + \ldots + \xi_n a_n \in S + x},
    \]
    where $S + x = \{s + x : s \in S\}$.
\end{definition}

    In these terms, $\rho(A) = \rho(A, \{0\})$. This general definition turns out to be convenient for us, due to the following simple properties.

    \begin{fact} \label{rho(A S)-1}
        Let $S$ be a subset of $V$. If $A'$ is a subsequence of $A$ then $\rho(A, S) \le \rho(A', S)$.
    \end{fact}
    \begin{proof}
        Let $A' = A[I]$, and fix $x \in V$. Denote $X = \sum_{i \in I} \xi_i a_i$, $Y = \sum_{i \notin I} \xi_i a_i$. Conditioning on the outcome of $Y$, we have
        \[
        \prob{X + Y \in S - x} = \EE_{Y}\Big[\prob{X \in S - x - Y \mid Y}\Big] \le \sup_{z \in V} \prob{X \in S - z} = \rho(A', S).\qedhere
        \]
    \end{proof}
    
    \begin{fact} \label{rho(A S)-2}
        Let $S_1, S_2$ be two subsets of $V$. Then 
        \[
        \max(\rho(A, S_1), \rho(A, S_2)) \le \rho(A, S_1 \cup S_2) \le \rho(A, S_1) + \rho(A, S_2).
        \]
    \end{fact}
    \begin{proof}
        The first inequality holds trivially. For the second one, denote $X = \xi_1 a_1 + \ldots + \xi_n a_n$. Then for any $x \in V$
        \[
        \prob{X \in (S_1 \cup S_2) - x} = \prob{X \in (S_1 - x) \cap (S_2 - x)}\le \sup_{x \in V} \prob{X \in S_1 - x} + \sup_{x \in V} \prob{X \in S_2 - x}. \qedhere
        \]
    \end{proof}
    
    \begin{fact} \label{rho(A S)-3}
        Let $\pi: V \to U$ is a surjective linear map. Then for any sequence $A = (a_1, \ldots, a_n)$ of vectors in $V$, and any set $S \subseteq U$, we have
        \[
        \rho(\pi(A), S) = \rho(A, \pi^{-1}(S)).
        \]
    \end{fact}
    \begin{proof}
        One can verify that for any outcomes of $\xi_1, \ldots, \xi_n$ and any $x \in V$ we have
        \[
        \xi_1\pi(a_1) + \ldots + \xi_n\pi(a_n) + \pi(x) \in S \quad \text{ if and only if } \quad \xi_1 a_1 + \ldots + \xi_n a_n + x \in \pi^{-1}(S).
        \]
        The statement then follows by taking the supremum over $x \in V$ of the probabilities of both these events.
    \end{proof}

    We also record the following ``dropping to a subspace'' lemma, which also appeared in \cite{kwan-sauermann-23} (in a slightly different form). Roughly speaking, it says that any sequence of nonzero vectors in a vector space $V$ contains a large subsequence, which has linear basis packing number inside a possibly smaller subspace $V' \subseteq V$.
    
    \begin{lemma}\label{dropping-to-subspace}
        Let $n, k$ and $b$ be positive integers, such that $n - (b-1) k(k+1)/2 > 0$. Consider a sequence of nonzero vectors $A = (a_1, \ldots, a_n)$ in a vector space $V$ of dimension $k$. Then there exists a linear subspace $V' \subseteq V$, and a subsequence $A'$ of $A$ of size at least $n - (b-1) k(k+1)/{2}$, such that all elements of $A'$ lie in $V'$, and the basis packing number of $A'$ (as a sequence in $V'$) is at least $b$.
    \end{lemma}
    \begin{proof}
        We argue by induction on $k$. In the case $k = 1$ each of $n \ge b$ vectors forms a one-element basis.
        
        Let $b'$ be the basis packing number of $A$. If $b' \ge b$, then we are done. Otherwise, let $I_1, \ldots, I_{b'}$ be the disjoint sets of indices corresponding to $b' \le b-1$ bases. Then the subsequence $A_0 = A[[n] \setminus (I_1 \cup \ldots \cup I_{b'})]$ of size $n - b'k$ does not contain a basis of $V$. Therefore, its vectors lie in a linear subspace $V_0 \subsetneq V$ of dimension $k-1$.        
        As
        \[
        n - b'k - (b-1)\frac{(k-1)k}{2} \ge n - (b-1)\frac{k(k+1)}{2} > 0, 
        \]
        we can apply the induction hypothesis to the sequence $A_0$ in the vector space $V_0$ to obtain the desired subsequence.
    \end{proof}

    Note that $\FF$ (which is a subfield of $\CC$) contains the integers. Much of our analysis will focus on counting lattice points within subsets of $\FF^k$, so we introduce the following notation.

    \begin{definition} \label{def:density}
    For a set $S \subseteq \FF^k$ and a real number $B \ge 0$, we define the integer point counting function as
    \[
    N_S(B) = |\{(x_1, \ldots, x_k) \in \ZZ^k \cap S : |x_i| \le B \text{ for } 1 \le i \le k\}|,
    \]
    and the integer points density function as
    \[
    d_S(B) = \sup_{\varphi} \left(\frac{N_{\varphi(S)}(B)}{|\ZZ^k \cap [-B, B]^k|}\right) = \sup_{\varphi}\left(\frac{N_{\varphi(S)}(B)}{(2\lfloor B\rfloor+1)^k}\right),
    \]
    where the supremum is taken over all bijective affine-linear maps $\varphi:\FF^k \to \FF^k$.
    \end{definition}
           
    Although slightly non-standard, this definition of $d_S(B)$ is convenient for our purposes. Note that $d_S(B)$ is invariant under bijective affine-linear transformations of $S$. Therefore, it does not depend on the choice of basis in $\FF^k$, and makes sense for a set $S$ in an abstract finite-dimensional vector space $V$ over $\FF$.
        
    Furthermore, we observe that it is also invariant under taking preimages of projections.

    \begin{proposition} \label{slicing}
        Let $\psi: \FF^r \to \FF^k$ be a surjective affine-linear map. Then for any set $S \subseteq \FF^k$ and any $B \ge 0$ we have
        \[
        d_{\psi^{-1}(S)}(B) = d_S(B).
        \]
    \end{proposition}
    \begin{proof}
        By replacing $B$ with its integer part $\lfloor B \rfloor$, we can assume that $B$ is a non-negative integer. 
        
        First, we prove the inequality $d_{\psi^{-1}(S)}(B) \ge d_S(B)$.
        Given a bijective affine-linear map $\varphi_1:\FF^k \to \FF^k$, there exists a bijective affine-linear map $\varphi_2:\FF^r \to \FF^r$ such that $\varphi_1 \circ \psi \circ \varphi_2 = p$, where $p:\FF^r \to \FF^k$ is the projection onto the first $k$ coordinates. Then 
        \[
        \varphi_2^{-1}(\psi^{-1}(S)) = p^{-1}(\varphi_1(S)),  
        \] 
        and therefore 
        \[
        (2B+1)^r d_{\psi^{-1}(S)}(B) \ge N_{\varphi_2^{-1}(\psi^{-1}(S))}(B) = N_{p^{-1}(\varphi_1(S))}(B) = (2B+1)^{r-k} N_{\varphi_1(S)}(B).
        \]
        Taking the supremum over $\varphi_1$ gives the desired inequality.
        
        Next, we prove the converse inequality $d_{\psi^{-1}(S)}(B) \le d_S(B)$. Let  $e_1, \ldots, e_r$ be the standard basis vectors of $\FF^r$. Consider a bijective affine-linear map $\varphi_2:\FF^r \to \FF^r$. The kernel of $\psi \circ \varphi_2$ has dimension $r-k$, thus we can choose $k$ standard basis vectors $e_{j_1}, \ldots, e_{j_k}$ such that the subspace $W$ spanned by them has trivial intersection with this kernel. 
                
        Then the restriction of $\varphi_2 \circ \psi$ to each translate of $W$ is a bijective affine-linear map. We consider ``slices'' of the box $[-B, B]^r$ by the translates of $W$, and bound the number of integer points on each of them in terms of $d_S(B)$.

        Let $J = \{j_1, \ldots, j_k\} \subseteq [r]$. Then
        \begin{align*}
        N_{\varphi_2^{-1}(\psi^{-1}(S))}(B) &= \Big|\Big\{(c_j)_{j \in [r]} \in \ZZ^r : |c_j| \le B, \; \sum_{j \in [r]} c_j e_j \in \varphi_2^{-1}(\psi^{-1}(S))\Big\}\Big| \\
        &= \sum_{\substack{(c_j)_{j \in [r] \setminus J} \in \ZZ^{r-k},\\ |c_j| \le B}} \Big|\Big\{(c_j)_{j \in J} \in \ZZ^k : |c_j| \le B, \; \sum_{j \in J} c_j e_j \in W \cap \Big(\varphi_2^{-1}(\psi^{-1}(S)) - \sum_{j \in [r] \setminus J}c_j e_j\Big)\Big\}\Big| \\
        &= \sum_{\substack{(c_j)_{j \in [r] \setminus J} \in \ZZ^{r-k},\\ |c_j| \le B}} \Big|\Big\{(c_j)_{j \in J} \in \ZZ^k : |c_j| \le B, \; \sum_{j \in J} c_j e_j \in \varphi_{(c_j)}^{-1}(S)\Big\}\Big|,
        \end{align*}
        where $\varphi_{(c_j)}: W \to \FF^k$ is the bijective affine-linear map defined by $\varphi_{(c_j)}(w) = \psi(\varphi_2(w + \sum_{j \in [r] \setminus J}c_j e_j$)). 
        Recalling the definition of the density function $d_S$, we conclude that 
        \[
        N_{\varphi_2^{-1}(\psi^{-1}(S))}(B) \le \sum_{\substack{(c_j)_{j \in [r] \setminus J} \in \ZZ^{r-k},\\ |c_j| \le B}} N_{\varphi_{(c_j)}^{-1}(S)}(B) \le (2B+1)^{r-k} \cdot (2B+1)^k d_S(B) = (2B+1)^r d_S(B).
        \]
        Taking the supremum over $\varphi_2$ implies that $d_{\psi^{-1}(S)}(B) \le d_S(B)$, completing the proof.
    \end{proof}

    We also note that the value of $d_S(B)$ does not ``jump too much'' when $B$ changes: namely, if $B_1$ and $B_2$ differ by at most a multiplicative constant factor, then the same holds true for $d_S(B_1)$ and $d_S(B_2)$.

    \begin{proposition} \label{no-jumps}
        Suppose that $0 \le B_1 \le B_2 \le cB_1$ for some $c \ge 1$. Then for any set $S \subseteq \FF^k$ we have
        \[
        \frac{1}{(3c)^k} d_S(B_1) \le d_S(B_2) \le 2^k d_S(B_1).
        \]        
    \end{proposition}
    \begin{proof}
        The first inequality follows by observing that $N_{\varphi(S)}(B_1) \le N_{\varphi(S)}(B_2)$ and 
        \[
        2\lfloor B_2 \rfloor + 1 \le 2B_2+1 \le c(2B_1 + 1) \le 3c (2\lfloor B_1 \rfloor + 1).
        \]
        For the second inequality, we cover the integer points of the box $[-B_2, B_2]^k$ by $M = \big\lceil(2\lfloor B_2\rfloor+1)/(2\lfloor B_1\rfloor + 1)\big\rceil^k$ boxes with side lengths $2\lfloor B_1 \rfloor$ centered at points with integer coordinates. Then for any bijective affine-linear map $\varphi: \FF^k \to \FF^k$ we have 
        \[
        N_{\varphi(S)}(B_2) \le M \cdot \sup_{x \in \ZZ^k}N_{\varphi(S)-x}(B_1) \le M \cdot (2 \lfloor B_1 \rfloor+1)^k d_S(B_1) \le (2(2\lfloor B_2 \rfloor+1))^k d_S(B_1).
        \]
        Thus, $N_{\varphi(S)}(B_2)/(2 \lfloor B_2 \rfloor + 1)^k \le 2^k d_S(B_1)$, and taking the supremum over $\varphi$ completes the proof.
    \end{proof}

    \section{Reduction to lattice point counting}
    \label{sec:lattice-point-counting}
    In this section we prove \cref{lattice-points-reduction}, stated below. For a general set $S$, this theorem provides an upper bound on the maximum $S$-translate probability $\rho(A,S)$ in terms of the integer point density function from \cref{def:density}, under the assumption that the maximum point probability $\rho(A)$ is polynomially large (this is the ``concentrated'' case described in \cref{sec:outline}).
    
    \begin{theorem} \label{lattice-points-reduction}
    Fix $\delta \in (0, 1)$ and $C, C_1 > 0$. Let $A = (a_1, \ldots, a_n)$ be a sequence of vectors in a finite-dimensional vector space $V$ (over a field $\FF \subseteq \CC$), such that the basis packing number of $A$ is at least $\delta n$ and $\rho(A) \ge n^{-C}$. Then there exists $r = O_C(1)$ such that for any subset $S \subseteq V$ we have 
    \begin{equation} \label{eq:lattice-points-reduction}
    \rho(A, S) \le O_{\delta, C, C_1}\left(d_S(\sqrt{n \log n}) \cdot (\log n)^r + n^{-C_1}\right).
    \end{equation}
    \end{theorem}

    \cref{example-lattice} shows that the bound in this theorem is sharp up to logarithmic factors.

    \begin{definition} \label{def:GAP}
    A subset $Q$ of an abelian group $G$ is called a \textit{proper symmetric generalized arithmetic progression} (\textit{proper symmetric GAP}, for short) of rank $r$ if there exist $v_1, \ldots, v_r \in G$ and $q_1, \ldots, q_r \in \NN$ such that
    \[
    Q = \{c_1v_1 + \ldots + c_rv_r : c_i \in \ZZ, \; |c_i| \le q_i \text{ for } 1 \le i \le r\},
    \]
    and each element of $Q$ can be represented as $c_1v_1 + \ldots + c_rv_r$ in a unique way. 
    \end{definition}

    A key ingredient in the proof of \cref{lattice-points-reduction} is an \emph{inverse theorem} for the linear Littlewood--Offord problem. 
    The first theorem of this kind was proved in seminal work of Tao and Vu \cite{tao-vu-09}; it states that if $\rho(A) \ge n^{-C}$, then almost all of the elements of $A$ are contained in a common GAP whose volume is at most $n^B$ and whose rank is at most $r$, for some $r$ and $B$ depending only on $C$. The quantitative aspects of this theorem were subsequently improved in theorems of Tao and Vu \cite{tao-vu-10} and Nguyen and Vu \cite{nguyen-vu-11}; we state the latter theorem below.

    \begin{theorem}[{Optimal inverse Littlewood--Offord theorem; Nguyen and Vu~\cite[Theorem 2.5]{nguyen-vu-11}}] \label{inverse}    
    Fix $\eps \in (0, 1)$ and $C > 0$. Let $A$ be a sequence of $n$ elements of an abelian torsion-free group, satisfying $\rho(A) \ge n^{-C}$.
    Then for any $n^\varepsilon \le s \le n$, there exists a proper symmetric GAP $Q$ of rank $r = O_{C, \varepsilon}(1)$, such that it contains all but at most $s$ elements of $A$, and
    \[
    |Q| = O_{C, \varepsilon} \left( \rho(A)^{-1} s^{-r/2} \right).
    \]
    \end{theorem}

    A significant shortcoming of \cref{inverse} is that the bound on $|Q|$ is in terms of $\rho(A)$, which can be much smaller than $\rho(A')$ (we wish to ``discard'' the elements in $A\setminus A'$, and it is important to avoid a dependence on the discarded elements). We can address this issue by \emph{iterating} \cref{inverse}, yielding the following result.

    \begin{theorem} \label{iterated-inverse}
    Fix $\eps \in (0, 1)$ and $C > 0$. Let $A$ be a sequence of $n$ elements of an abelian torsion-free group, satisfying $\rho(A) \ge n^{-C}$.
    Then for any $n^\varepsilon \le s_1 \le n$, there exists a proper symmetric GAP $Q$ of rank $r = O_{C, \varepsilon}(1)$ and a subsequence $A'$ of $A$ of size at least $n - s_1$, such that all elements of $A'$ lie in $Q$, and
    \[
    |Q| = O_{C, \varepsilon} \left( \rho(A')^{-1} (s_1/\log n)^{-r/2} \right).
    \]
    \end{theorem}

    \begin{remark*}
    We suspect that in the setting of \cref{iterated-inverse}, a stronger bound of the form $O_{C, \varepsilon} (\rho(A)^{-1} s_1^{-r/2})$ might hold (this would yield a common generalisation of \cref{inverse,iterated-inverse}).
    \end{remark*}

    \begin{proof}[\textbf{Proof of \cref{iterated-inverse}}]
        We assume that $n$ is sufficiently large compared to $C$ and $\eps$. By decreasing $s_1$, we may assume that $s_1 \le n/2$. Then by increasing $C$, we may also assume that $\rho(A) \ge (n-s_1)^{-C}$. Let $L = \lceil C \log_2 n \rceil$, and let $s = s_1/L$.

        We construct a descending chain $A_0, A_1, \ldots, A_L$ of subsequences of $A$ with sizes $n_0 \ge n_1 \ge \ldots \ge n_L$ satisfying $n_i \ge n - is$, as follows. Set $A_0 = A$. To obtain $A_{i+1}$ from $A_i$, first note that $n^{\eps/2} \le s \le n$, and that by \cref{rho(A S)-1} 
        \[
        \rho(A_i) \ge \rho(A) \ge (n-s_1)^{-C} \ge (n - is)^{-C} \ge n_i^{-C}.
        \]
        Therefore, we can apply the optimal inverse theorem (\cref{inverse}) to the sequence $A_i$. As a result, we obtain a proper symmetric GAP $Q_i$ of rank $r_i = O_{C, \varepsilon}(1)$, such that all but at most $s$ elements of $A_i$ lie in $Q_i$, and
        \begin{equation} \label{eq:|Q_i|}
        |Q_i| = O_{C, \varepsilon} \left(\rho(A_i)^{-1} s^{-r_i/2} \right).
        \end{equation}
        Let $A_{i+1}$ be the subsequence of $A_i$ consisting of all elements that lie in $Q_i$. Then we indeed have
        \[
        n_{i+1} \ge n_i - s \ge n - (i+1)s.
        \]

        Suppose that for some $0 \le i < L$ we have $\rho(A_{i+1}) \le 2 \rho(A_i)$. In this case we can replace $\rho(A_i)$ by $\rho(A_{i+1})$ in the estimate \cref{eq:|Q_i|}. Then we are done by taking $A' = A_{i+1}$ and $Q = Q_i$.

        Otherwise, we have $\rho(A_{i+1}) > 2 \rho(A_i)$ for all $0 \le i < L$. Then
        \[
        \rho(A_L) > 2^L \rho(A) \ge 2^{C \log_2 n} n^{-C} = 1.
        \]
        But $\rho(A_L)$ is a supremum of probabilities, a contradiction. 
    \end{proof}

    We will also need a simple concentration inequality (a consequence of Hoeffding's inequality \cite{hoeffding-63}). For positive reals $q_1, \ldots, q_r$ we let $Q_r(q_1, \ldots, q_r)=\{(x_1, \ldots, x_r) \in \ZZ^r : |x_i| \le q_i\}$.

    \begin{proposition} \label{Chernoff}
        For any vectors $a_1, \ldots, a_m$ in $Q_r(q_1, \ldots, q_r)$ and $t > 0$ we have 
        \[
        \prob{\xi_1 a_1 + \ldots + \xi_m a_m \notin Q_r(t q_1, \ldots, t q_r)} \le 2r \exp\left(-\frac{t^2}{2m}\right).
        \]
    \end{proposition}
    \begin{proof}
        For each $1 \le j \le r$, the $j$-th coordinate of $\xi_1 a_1 + \ldots + \xi_m a_m$ is a sum of $m$ independent random variables; each of them has expected value equal to zero and is supported in $[-q_j, q_j]$. By Hoeffding's inequality, 
        \[
        \prob{|(\xi_1 a_1 + \ldots + \xi_m a_m)_j| > tq_j} \le 2 \exp\left(-\frac{2 \cdot (t q_j)^2}{m (2q_j)^2}\right) \le 2 \exp\left(-\frac{t^2}{2m}\right).
        \]
        The union bound over $1 \le j \le r$ completes the proof. 
    \end{proof}

    \begin{proof}[\textbf{Proof of \cref{lattice-points-reduction}}]
        We apply \cref{iterated-inverse} to $A$ with $s_1 = \delta n / 2$. As a result, we obtain a subsequence $A'$ of size at least $(1-\delta/2)n$ and a proper symmetric GAP $Q$ of rank $r = O_C(1)$ such that all elements of $A'$ lie in $Q$, and 
        \[
        |Q| \le O_{\delta, C}\left(\rho(A')^{-1} (n/\log n)^{-r/2}\right).
        \]
        Let $v_1, \ldots, v_r$ be the generators of $Q$, and let $e_1, \ldots, e_r$ be the standard basis vectors of $\FF^r$. Let $\psi: \FF^r \to V$ be the linear map defined by $\psi(e_i) = v_i$ for all $i$. Then $Q$ is the image of the set of integer points of some box $Q_r(q_1, \ldots, q_r)$ under this map $\psi$. Since $Q$ is proper, the restriction $\psi|_{Q_r(q_1, \ldots, q_r)}$ provides a bijection between $Q_r(q_1, \ldots, q_r)$ and $Q$. In particular, 
        \begin{equation} \label{eq:|Q_r|}
        |Q_r(q_1, \ldots, q_r)| = |Q| \le O_{\delta, C}\left(\frac{(\log n)^{r/2}}{\rho(A') \cdot n^{r/2}}\right).
        \end{equation}
        
        As the vectors of $A'$ lie in $Q$, we can define $A^*$ to be the sequence of vectors in $\CC^r$ obtained by taking preimages of the elements of $A'$ under this bijection.
        
        By assumption, the basis packing number of $A$ is at least $\delta n$. Hence, the basis packing number of $A'$ is at least $\delta n /2 > 0$. In particular, the sequence $A' = \psi(A^*)$ contains a basis of $V$, thus the map $\psi$ is surjective.  

        Then by \cref{rho(A S)-1,rho(A S)-3} we have
        \begin{equation} \label{eq:rho(A S)}
        \rho(A, S) \le \rho(A', S) = \rho(A^*, \psi^{-1}(S)),
        \end{equation}
        and (by \cref{rho(A S)-2,rho(A S)-3})
        \begin{equation} \label{eq:rho(A^*)}
        \rho(A^*) \le \rho(A^*, \ker \psi) = \rho(A').
        \end{equation}

        Let $a^*_1, \ldots, a^*_m$ be the vectors of the sequence $A^*$, where $m \ge (1 - \delta/2)n$, and consider $x \in \FF^r$. By \cref{eq:rho(A S)}, it suffices to estimate the probability that $\xi_1 a^*_1 + \ldots + \xi_m a^*_m$ lies in the set $S^*_x = (\psi^{-1}(S) - x)$.
        
        By increasing $C_1$, we may assume that $\sqrt{2C_1}$ is a positive integer. Define the dilated box
        \[
        Q^* = Q_r(\sqrt{2C_1 n \log n} \cdot q_1, \ldots, \sqrt{2C_1 n \log n} \cdot q_n).
        \]
        By \cref{Chernoff}, combined with the fact that $n/2 \le m \le n$, we have
        \[
        \prob{\xi_1 a^*_1 + \ldots + \xi_m a^*_m \notin Q^*} \le 2r \cdot m^{-C_1} \le O_{C, C_1}(n^{-C_1}).
        \] 
        This corresponds to the $n^{-C_1}$ term in the desired bound \cref{eq:lattice-points-reduction}.

        For the rest of the argument we focus on the probability that $\xi_1 a^*_1 + \ldots + \xi_m a^*_m$ lies in the set $S^*_x \cap Q^*$. We estimate this probability by the union bound: that is, we view it as the sum of $\prob{\xi_1 a^*_1 + \ldots + \xi_m a^*_m  = y}$ over all $y \in S^*_x \cap Q^*$. 
        By definition, $\rho(A^*)$ is the maximum point probability. Then \cref{eq:rho(A^*)} implies that for any $y \in \FF^r$
        \begin{equation} \label{eq:point-probability}
        \prob{\xi_1 a^*_1 + \ldots + \xi_m a^*_m  = y} \le \rho(A^*) \le \rho(A').
        \end{equation}

        To estimate the number of points in $|S^*_x \cap Q^*|$, we partition $Q^*$ into boxes with side lengths $2\sqrt{n \log n}$. The number of points of $S^*_x$ inside each constituent box may be expressed as a product of the total number of integer points in this box and the proportion of them that lie in $S_x^*$. This proportion, in turn, is bounded above by the relevant value of the density function $d_{S^*_x}(\sqrt{n \log n})$. Recalling that $S^*_x = \psi^{-1}(S) - x$ and $\psi$ is a surjective linear map, by \cref{slicing} we have $d_{S^*_x} = d_S$.
        
        Taking the sum over all boxes in our partition, we obtain
        \begin{equation} \label{eq:|S^*_x|}
        |S^*_x \cap Q^*| \le |Q^*| \cdot d_{S^*_x}(\sqrt{n \log n}) = |Q^*| \cdot d_S(\sqrt{n \log n}).
        \end{equation}
        From \cref{eq:|Q_r|} we have
        \begin{equation} \label{eq:|Q^*|}
        |Q^*| \le (n \log n)^{r/2} \cdot |Q_r(q_1, \ldots, q_r)| \le O_{\delta, C, C_1}\left(\rho(A')^{-1}(\log n)^r\right).
        \end{equation}
        Finally, we combine the inequalities \cref{eq:point-probability}, \cref{eq:|S^*_x|} and \cref{eq:|Q^*|} to conclude that
        \begin{align*}
        \prob{\xi_1 a^*_1 + \ldots + \xi_m a^*_m \in S^*_x \cap Q^*}
        &\le \rho(A') \cdot |S^*_x \cap Q^*|
        \le \rho(A') \cdot |Q^*| \cdot d_{S}(\sqrt{n \log n}) \\
        &\le O_{\delta, C, C_1}\left(d_{S}(\sqrt{n \log n}) \cdot (\log n)^r \right).
        \end{align*}
        This completes the proof.
    \end{proof}

    \section{Sets of points in convex position: proof of \texorpdfstring{\cref{convex-position}}{Theorem~\ref{convex-position}}}\label{sec:convex}
    In this section we prove \cref{convex-position} and deduce \cref{convex-1/2}. As described in \cref{sec:outline}, our strategy to prove \cref{convex-position} is to reduce the problem to counting integer points in a certain preimage of $S$. Therefore, we need an estimate on the maximum possible number of integer points in convex position inside $[-B, B]^k$. The following theorem is due to Andrews~\cite{andrews-63} (generalising an earlier result of Jarn\'ik~\cite{jarnik-26} for the case $k=2$).

    \begin{theorem}[{Andrews \cite{andrews-63}; see also \cite[Theorem 2]{barany-larman-98}}] \label{Andrews}
        Let $S \subseteq \RR^k$ be a set of points in convex position. Then for any $B \ge 1$
        \[
        N_S(B) \le O_k(B^{k - \frac{2k}{k+1}}).
        \]
    \end{theorem}

    \begin{remark*} \cref{Andrews} can be viewed as an upper bound on the number of vertices of a lattice polytope contained in $[-B, B]^k$. B\'ar\'any and Larman \cite[Theorem 1]{barany-larman-98} proved that this bound is tight up to a multiplicative constant factor: a lower bound of the same order of magnitude is achieved by the convex hull of integer points inside the ball of radius $B$.
    \end{remark*}

    As described in the outline, we use the following result of Fox, Kwan and Spink \cite{fox-kwan-spink-23} to handle the ``spread-out'' case. 

    \begin{theorem}[Fox, Kwan and Spink {\cite[Theorem 1.9(2)]{fox-kwan-spink-23}}] \label{spread-out}
    Let $S \subseteq \RR^k$ be a set of points in convex position, and let $A$ be a sequence of nonzero vectors in $\RR^k$. Then
    \[
    \rho(A, S) \le O_k\left(\rho(A)^{1/(k2^{k-1})}\right).
    \]
    \end{theorem}

    \begin{proof}[\textbf{Proof of \cref{convex-position}}]
        We have a sequence $A$ of vectors in $\RR^k$ with the basis packing number at least $b$. Consider the subsequence $A_0 = A[I_0]$ containing only the vectors of $b$ disjoint bases. It has size $m := bk$ and basis packing number equal to $b$. 
        
        We need to estimate the probability that $\xi_1 a_1 + \ldots + \xi_n a_n$ lies in $S$. It is bounded by $\rho(A, S)$, which is at most $\rho(A_0, S)$ by \cref{rho(A S)-1}.

        Let $C = k2^{k-1}$. First, suppose that $\rho(A_0) < m^{-C}$. Then, by \cref{spread-out}, we have
        \[
        \rho(A_0, S) \le O_k\left(\rho(A_0)^{1/(k2^{k-1})}\right),
        \] 
        which is at most $O_k(m^{-1})$. Since $m = bk$, this gives the desired bound.
        
        Therefore, we may assume that $\rho(A_0) \ge m^{-C}$. Applying \cref{lattice-points-reduction} to $A_0$ with $\FF = \RR$, $\delta = 1/k$ and $C_1 = 1$, we obtain
        \begin{equation} \label{eq:rho(A_0 S)}
        \rho(A_0, S) \le O_{k}\left(d_S(\sqrt{m \log m}) \cdot (\log m)^r + m^{-1}\right)
        \end{equation}
        for some $r = O_k(1)$.        
        Observe that for any bijective affine-linear map $\varphi:\RR^k \to \RR^k$ the set $\varphi(S)$ is also in convex position. Thus, by the definition of the density function $d_S$ combined with Andrews' theorem (\cref{Andrews}), for any $B \ge 1$ we have
        \[
        d_S(B) = \sup_{\varphi}\left(\frac{N_{\varphi(S)}(B)}{(2\lfloor B \rfloor+1)^k}\right) \le O_k(B^{-\frac{2k}{k+1}}).
        \]
        We substitute this into \cref{eq:rho(A_0 S)}, recalling that $m = bk$, to conclude that
        \[
        \rho(A_0, S) \le O_k(b^{-\frac{k}{k+1}} (\log b)^{r-\frac{k}{k+1}}).
        \]
        This completes the proof.
    \end{proof}

    Next, we combine \cref{convex-position} with the ``dropping to a subspace'' argument (\cref{dropping-to-subspace}) to deduce \cref{convex-1/2}.
    \begin{proof}[\textbf{Proof of \cref{convex-1/2}}]
        Fix an arbitrary $\eps \in (0, 1)$. We will prove that if $n$ is sufficiently large in terms of $k$ and $\eps$ then $\prob{\xi_1 a_1 + \ldots + \xi_n a_n \in S} \le (2\sqrt{2/\pi} + \eps) n^{-1/2}$.

        We apply \cref{dropping-to-subspace} to the sequence $A = (a_1, \ldots, a_n)$ with $b = \lfloor \eps n/(k(k+1))\rfloor + 1$. As a result, we obtain a subsequence $A' = A[I]$ of size at least $(1-\eps/2)n$, such that all elements of $A'$ lie in a linear subspace $V' \subseteq \RR^k$, and the basis packing number of $A'$ inside $V'$ is at least $b = \Omega_{k, \eps}(n)$.
        Conditioning on the outcomes of the random variables $(\xi_i)_{i \in [n] \setminus I}$, we have
        \[
        \prob{\xi_1 a_1 + \ldots + \xi_n a_n \in S} \le \sup_{x \in \RR^k} \prob{\sum_{i \in I} \xi_i a_i \in (S - x) \cap V'}. 
        \]
        Fix an arbitrary $x \in \RR^k$, and let $\ell = \dim V'$. First we consider the case $\ell \ge 2$. Applying \cref{convex-position} to $A'$ and $(S-x)\cap V'$, we conclude that
        \[
        \prob{\sum_{i \in I} \xi_i a_i \in (S - x) \cap V'} \le O_{\ell}\left(b^{-\frac{\ell}{\ell+1}} (\log b)^{C_{\ell}}\right). 
        \] 
        Since $2 \le \ell \le k$ and $b = \Omega_{k, \eps}(n)$, this bound is at most $2\sqrt{2/\pi} \cdot n^{-1/2}$ when $n$ is sufficiently large (in terms of $k$ and $\eps$).

        Next, we deal with the case $\ell = 1$. In this case, $(S - x) \cap V'$ is a set of points in convex position on a line. Then it contains at most $2$ points, and the desired probability can be bounded by the classical Erd\H{o}s--Littlewood--Offord theorem:
        \[
        \prob{\sum_{i \in I} \xi_i a_i \in (S - x) \cap V'} \le 2 \cdot 2^{-|I|}\binom{|I|}{\big\lfloor|I|/2\big\rfloor} = \left(2\sqrt{2/\pi} + o(1)\right)|I|^{-1/2}.
        \]
        Since $|I| \ge (1-\eps/2)n$, one can check that this expression is at most $(2\sqrt{2/\pi} + \eps)n^{-1/2}$ when $n$ is sufficiently large (in terms of $\eps$).
    \end{proof}
    
    \section{Algebraic preliminaries} 
    \label{sec:algebraic-preliminaries}
    Now, for the rest of the paper we turn our attention to \cref{varieties-1/d} and its corollaries. We start with some preliminaries from algebraic geometry and number theory.

    \subsection{Algebraic geometry}

    In this subsection, we review some basic concepts and facts from algebraic geometry. We loosely follow the exposition in \cite[Section 7.1]{cohen-moshkovitz-23}, and refer to \cite[Chapter 1]{hartshorne-77} for more details.

    An \textit{affine algebraic variety} $S \subseteq \CC^k$ (a \textit{variety}, for short) is the set of common zeros of a finite collection of polynomials $f_1, \ldots, f_m \in \CC[x_1, \ldots, x_k]$:
    \[
    S = \{(x_1, \ldots, x_k) \in \CC^k : f_1(x_1, \ldots, x_k) = \ldots = f_m(x_1, \ldots, x_k) = 0\}.
    \]
    A variety is called \textit{irreducible} if it cannot be written as a union of two proper subvarieties. Each variety $S$ can be uniquely written as the union of irreducible subvarieties $S_1 \cup \ldots \cup S_m$, such that $S_i \not\subseteq S_j$ for any $i \neq j$. Varieties $S_1, \ldots, S_m$ are called the \textit{irreducible components} of $S$.
    
    The \textit{dimension} of an irreducible variety $S$ is the maximal integer $\ell$ such that there exists a chain of non-empty irreducible subvarieties $S_0 \subsetneq S_1 \subsetneq \ldots \subsetneq S_\ell = S$. The dimension $\dim S$ of an arbitrary variety $S$ is defined as the maximum dimension of its irreducible components. By convention, the empty variety is reducible and has dimension $-\infty$.

    The \textit{codimension} $\codim S$ of a variety $S \subseteq \CC^k$ is defined as $k - \dim S$. 

    The \textit{degree} of an irreducible variety $S$ is the cardinality of the intersection of $S$ with a ``generic'' affine subspace of dimension $\codim S$ (a well-defined positive integer). The degree $\deg S$ of an arbitrary variety $S$ is defined as the sum of the degrees of its irreducible components. By convention, the degree of the empty variety is $0$.
    
    \begin{fact} \label{unions}
        For any two varieties $S, T \subseteq \CC^k$
        \[
        \deg(S \cup T) \le \deg(S) + \deg(T).
        \]
    \end{fact}

    \begin{fact} \label{projections}
        For any variety $S \subseteq \CC^k$ and a surjective affine-linear map $\pi: \CC^r \to \CC^k$ we have
        \[
        \codim \pi^{-1}(S) = \codim S, \quad \deg \pi^{-1}(S) = \deg S.
        \]
        Furthermore, if $S$ is irreducible, then $\pi^{-1}(S)$ is also irreducible.
    \end{fact}

    \begin{fact} \label{one-polynomial}
        Let $f \in \CC[x_1, \ldots, x_k]$ be an irreducible (over $\CC$) polynomial of degree $d \ge 1$. Then $S = \{x \in \CC^k : f(x) = 0\}$ is an irreducible variety of dimension $k-1$ and degree $d$.
    \end{fact}

    \begin{fact}[Generalized B\'ezout's theorem {\cite[Example 12.3.1]{fulton-84}; see also \cite{cohen-moshkovitz-23, bukh-tsimerman-12}}] \label{Bezout}       
        For any two varieties $S, T \subseteq \CC^k$,
        \[
        \deg(S \cap T) \le \deg(S) \cdot \deg(T).
        \]
    \end{fact}

    \begin{proposition} \label{infinite-intersection}
        Let $\{S_i\}_{i \in I}$ be a (not necessarily finite) collection of varieties in $\CC^k$. Suppose that $\deg(S_i) \le d$ for all $i \in I$. Then the set $S = \bigcap_{i \in I} S_i$ is a variety of degree at most $d^k$.
    \end{proposition}
    \begin{proof}
        We will prove a stronger statement: for any irreducible variety $T \subseteq \CC^k$ of dimension at most $\ell$ and degree at most $d_0$, the intersection $T \cap S$ is a variety of degree at most $d_0 d^{\ell}$. The proposition then follows from this statement applied with $T = \CC^k$.

        We argue by induction on $\ell = \dim T$. If $T \cap S = T$, there is nothing to prove. Otherwise, there exists $i \in I$ such that $T \cap S_i \subsetneq T$. Let $T_1, \ldots, T_m$ be the irreducible components of $T \cap S_i$. Each of them has dimension at most $\ell - 1$, and by B\'ezout's theorem (\cref{Bezout}) we have
        \[
        \sum_{j = 1}^m \deg(T_j) = \deg(T \cap S_i) \le d_0 d.
        \]
        By the induction hypothesis, we know that for each $1 \le j \le m$ the set $T_j \cap S$ is a variety, and that
        \[
        \deg(T_j \cap S) \le \deg(T_j) \cdot d^{\ell-1}.
        \] 
        As $T \cap S = \bigcup_{j = 1}^m (T_j \cap S)$, by \cref{unions} we conclude that
        \[
        \deg(T \cap S) \le \sum_{j = 1}^m \deg(T_j \cap S) \le d^{\ell-1}\sum_{j = 1}^m \deg(T_j) \le d_0 d^\ell. \qedhere
        \]
    \end{proof}

    \subsection{Number theory} \label{sec:number-theory}

    A large area of research in number theory is concerned with counting integer solutions to polynomial equations or, more generally, integer points on algebraic varieties. The most basic result in this direction is the Schwartz--Zippel lemma.

    \begin{proposition}[Schwartz--Zippel lemma for varieties, see for example {\cite[Lemma 14]{bukh-tsimerman-12}}]\label{trivial-bound}
        Let $S \subseteq \CC^k$ be a variety of dimension $\ell$ and degree $d$. Then for any $B \ge 1$,
        \[
        N_S(B) \le d \cdot (2B+1)^{\ell}.
        \]
    \end{proposition}    
    Since bijective affine-linear maps preserve dimension and degree (by \cref{projections}), in terms of the density function this means $d_S(B) \le d \cdot (2B+1)^{-(k-\ell)}$.

    \cref{trivial-bound} is sharp when $S$ is a union of $d$ axis-parallel affine subspaces of dimension $\ell$ (and it can be approximately sharp whenever $S$ contains a dimension-$\ell$ affine subspace as one of its irreducible components). However, one can obtain much stronger bounds by making certain assumptions on $S$.  We will need a few different results in this direction.
    
    First, we can make the assumption that $S$ is irreducible. The following bound in this setting was proved by Pila~\cite{pila-96} (improving on his slightly weaker bound in \cite{pila-95}). It was proved using the so-called \emph{Bombieri--Pila determinant method} (famously introduced by Bombieri and Pila~\cite{bombieri-pila-89} to prove a similar theorem for curves in $\mb R^2$). We refer the reader to \cite{bloom-lichtman-23} for a recent survey on this topic.

    \begin{theorem}[{Pila \cite{pila-95, pila-96}}] \label{Pila}
        Let $S \subseteq \CC^k$ be an irreducible variety of dimension $\ell$ and degree $d$. Then for any $B \ge 2$,
        \[
        N_S(B) = O_{d, k}\big(B^{\ell-1+1/d} (\log B)^{C_d}\big)
        \]
        for some constant $C_d$ depending only on $d$.
    \end{theorem}
    Again, \cref{projections} allows to rewrite this in terms of the density function as 
    \[
    d_S(B) \le O_{d, k}\big(B^{-(k-\ell+1-1/d)}(\log B)^{C_d}\big).
    \]
    
    The example $S = \{x \in \CC^k : x_1 = x_2^d\}$ shows that \cref{Pila} is sharp up to logarithmic factors, i.e., one cannot hope to remove the ``$1/d$'' term in the exponent, in general. However, there is a general belief that if one makes a mild assumption ruling out examples of this type, one should expect a bound of the form $N_S(B)\le B^{\ell-1+o(1)}$. A conjecture along these lines was first proposed by Heath-Brown~\cite{heath-brown-83}, and this conjecture (and its variants) are usually collectively referred to as the ``dimension growth conjecture''. Various partial results are available, see \cite{vermeulen-24} and the references therein.
    In this paper we use the following result, due to Vermeulen \cite{vermeulen-24} for $d \ge 4$ and due to Browning and Gorodnik~\cite{BG17} for $d = 2$ (using ideas of Browning, Heath-Brown and Salberger \cite{browning-heath-brown-salberger-06}; see also \cite{salberger-23-preprint}), which settles the (uniform) ``affine dimension growth conjecture'' for affine hypersurfaces of degree $d \neq 3$.

    \begin{theorem}[{Affine dimension growth conjecture for hypersurfaces; Vermeulen \cite[Theorem 1.2]{vermeulen-24} and Browning--Gorodnik \cite[Theorem 1.11]{BG17}}] \label{affine-dimension-growth}
    Consider an irreducible polynomial $f \in \CC[x_1, \ldots, x_k]$ of degree $d \neq 3$. Suppose that $f$ cannot be represented as a polynomial of two linear forms. Then, with $S\subseteq \CC^k$ as the zero set of $f$, and for any $B \ge 1$, $\eps > 0$, we have
    \[
    N_S(B) \le O_{d, k, \eps}(B^{k-2+\eps}).
    \]
    \end{theorem}

    Note that if the assumption of \cref{affine-dimension-growth} holds for a polynomial $f$, then it also holds for $f \circ \varphi$ for any bijective affine-linear map $\varphi:\CC^k \to \CC^k$. Therefore, we can rewrite the resulting bound in terms of the density function as $d_S(B) \le O_{d, k, \eps}(B^{-2+\eps})$.

    Vermeulen \cite{vermeulen-24} and Browning--Gorodnik \cite{BG17} state their results only for polynomials with rational coefficients. The reason is that this is the hardest case, which is also most natural to consider from number-theoretic point of view. Since we would like to apply \cref{affine-dimension-growth} to an arbitrary polynomial $f$, we provide a short argument which handles the case when $f$ is not proportional to a polynomial with coefficients in $\QQ$. Namely, \cref{galois-trick} below (applied with $\FF = \QQ$) implies that in this case the set of integer zeros of $f$ lies in a variety of codimension at least $2$. Then the desired bound on its size follows directly from the Schwartz--Zippel lemma (\cref{trivial-bound}).

    \begin{proposition} \label{galois-trick}       
       Consider a field $\mb F\subseteq \mb C$, and consider an irreducible polynomial $f \in \CC[x_1, \ldots, x_k]$ of degree $d$, which is not proportional to a polynomial with coefficients in $\FF$. Let $S \subseteq \CC^k$ be the zero set of $f$. Then there exists another variety $T \subseteq S$ of dimension at most $k-2$ and degree at most $d^2$ such that $T \cap \FF^k = S \cap \FF^k$.
    \end{proposition}
    \begin{proof}
        Rescale $f$ so that one of its coefficients is equal to $1$. Then it has a coefficient $z \in \CC\setminus\FF$.

        Recall the following simple fact: for any $z \in \CC \setminus \FF$ there exists an automorphism $\sigma$ of $\CC$ which acts as the identity on $\FF$ but does not fix $z$.
        To prove this fact, one can first define this automorphism on $\FF(z)$ by sending $z$ to a different root of the minimal polynomial of $z$ over $\FF$ if it is algebraic over $\FF$, and to (say) $z+1$ if $z$ is transcendental over $\FF$. Then one can extend this automorphism to the whole of $\CC$ (see for example \cite{yale-66}).

        Applying this automorphism $\sigma$ to each coefficient of $f$, we obtain a polynomial $f^\sigma$, which is not proportional to $f$ but still satisfies $f(q) = f^\sigma(q)$ for any $q \in \FF^k$. Therefore, the variety $T$ defined as
        \[
        T = \{x \in \CC^k : f(x) = f^\sigma(x) = 0\}
        \] 
        indeed satisfies $T \cap \FF^k = S \cap \FF^k$. By \cref{one-polynomial}, $T$ is an intersection of two distinct irreducible varieties of dimension $k-1$ and degree $d$. Then it has dimension at most $k-2$ and, by B\'ezout's theorem (\cref{Bezout}), degree at most $d^2$.
    \end{proof}

    \section{Decomposition into subspaces}
    \label{sec:decoupling}

    In this section we prove \cref{structure-randomness}, stated below. As outlined in \cref{sec:outline}, this provides a decomposition of $\mb C^k$ into a ``structured'' and subspace $W$ and a ``disordered'' subspace $U$.

    \begin{theorem} \label{structure-randomness}
        Fix $\delta, C_1 > 0$. Let $S \subseteq \CC^k$ be a variety of degree at most $d$. Consider a sequence $A$ of vectors in $\CC^k$ with basis packing number at least $\delta n$. 

        Then there exists a decomposition of $\CC^k$ as $U \oplus W$ for some linear subspaces $U$ and $W$, a variety $S' \subseteq W$ of degree at most $d^k$, and a subsequence $A'$ of $A$ satisfying all the following conditions:
        \begin{enumerate}
            \item[(1)] The basis packing number of $A'$ is at least $(\delta/(2(2k)^k))\cdot  n$;
            \item[(2)] Let $\pi_W: \CC^k \to W$ be the projection map. Then $\rho(\pi_W(A')) \ge n^{-C}$ for some constant $C$ not depending on $n$ (but possibly depending on $\delta, C_1, d, k)$;
            \item[(3)] $\pi_W^{-1}(S') \subseteq S$, and $\rho(A', S \setminus \pi_W^{-1}(S')) \le n^{-C_1}$.
        \end{enumerate}
    \end{theorem}

    Intuitively, 
    condition (2) says that the projection of (a subsequence of) $A$ onto $W$ has polynomially large point probabilities (which allows us to apply an inverse Littlewood--Offord theorem such as \cref{inverse}). The set $\pi_W^{-1}(S') \subseteq S$ can be viewed as ``the part of $S$ which we can control via its projection onto $W$''. Condition (3) then gives us control over the complementary part of $S$ (which cannot be studied via its projection onto $W$).

    Our strategy to prove \cref{structure-randomness} is to consider the following procedure. We begin with $A' = A$ and $U = \mathbb{C}^k$, where conditions (1) and (2) hold automatically. Then we show that the only way for condition (3) to fail is if $A'$ contains a ``linear-size'' subsequence $A_1$ for which $\rho(A_1, U_1)$ remains polynomial in $n$ for some proper subspace $U_1 \subsetneq U$. In that case, we set $A' = A_1$ and $U = U_1$ (while maintaining conditions (1) and (2)), and repeat the process. As the dimension of $U$ decreases on each step, the procedure terminates after at most $k$ steps.

    Later in this section, we will state and prove \cref{one-step}, which describes a single step of the above procedure. Its proof relies on \cref{iterated-decoupling}, stated below, which allows one to bound the ``variety probability'' $\rho(A, S)$ in terms of certain ``subspace probabilities'' $\rho(A', V)$ (for subsequences $A'$ of $A$ and subspaces $V$ contained in a translate of $S$). 
    
    \begin{lemma} \label{iterated-decoupling}
        Let $S \subseteq \CC^k$ be a variety of dimension at most $\ell$ and degree at most $d$. Consider a sequence $A$ of vectors in $\CC^k$, partitioned into $\ell+1$ subsequences $A_0, \ldots, A_\ell$. Then
        \[
        \rho(A, S) \le (\ell+1) \cdot d \cdot \left(\sup_{i, V} \rho(A_i, V)\right)^{1/2^{\ell}},
        \]
        where the supremum is taken over $0 \le i \le \ell$ and over linear subspaces $V \subseteq \CC^k$, such that $V \subseteq S - y$ for some $y \in \CC^k$.
    \end{lemma}

    The proof of \cref{iterated-decoupling} is based on an ``iterative decoupling argument'', inspired by the approach of \cite{kwan-sauermann-23}. For context, \emph{decoupling} is a general term for a large body of techniques in probability and statistics for ``reducing from dependent situations to independent ones'' (see for example the monograph~\cite{dlPG12}). In Littlewood--Offord theory, 
    ``decoupling'' usually refers to a class of techniques to reduce polynomial anticoncentration to linear anticoncentration (popularised by Costello, Tao and Vu~\cite{costello-tao-vu-06}), via inequalities such as \cref{decoupling-lemma} below. The particular statement of \cref{decoupling-lemma} appears (for example) as \cite[Lemma 8.4]{costello-vu-08}, but for the convenience of the reader we provide the short proof.

    \begin{lemma} \label{decoupling-lemma}
        If an event $\mc E(X,Y)$ depends on independent random objects $X,Y$, and $X'$ is an independent copy of $X$, then
        \[
        \prob{\mc{E}(X, Y)} \le \Big(\prob{\mc{E}(X, Y) \text{ and } \mc{E}(X', Y)}\Big)^{1/2}.
        \]
    \end{lemma}

    \begin{proof}
        By the Cauchy--Schwarz inequality, we have
        \begin{align*}
        \prob{\mc E(X,Y)\text{ and }\mc E(X',Y)} &= \EE_Y\!\Bigl[\prob{\mc E(X,Y)\text{ and }\mc E(X',Y)\,|\, Y}\Bigr]=\EE_Y\!\Bigl[\prob{\mc E(X,Y)\,|\, Y}^2\Bigr]\\
        &\ge \EE_Y\!\Bigl[\prob{\mc E(X,Y)\,|\, Y}\Bigr]^2=\prob{\mc E(X,Y)}^2.
        \end{align*}
        Taking square roots on both sides completes the proof.
    \end{proof}

    \begin{proof}[\textbf{Proof of \cref{iterated-decoupling}}]
        Let $S_1, \ldots, S_m$ be the irreducible components of $S$. Recall that (by definition) $\deg S = \sum_{j = 1}^m \deg S_j$, and that (by \cref{rho(A S)-2}) $\rho(A, S) \le \sum_{j = 1}^m \rho(A, S_j)$. Therefore, by treating each irreducible component separately, we may assume that $S$ is irreducible.
        
        We argue by induction on $\ell$. In the base case $\ell = 0$ we consider only one subsequence $A_0 = A$, the variety $S$ consists of a single point, and the only linear subspace appearing in the supremum has dimension zero. So, in this case both sides of the inequality are equal to $\rho(A)$.

        Let $n$ be the size of $A$, and let $I_0 \subseteq [n]$ be the set of indices corresponding to the subsequence $A_0$. Consider independent random variables 
        \[
        X = \sum_{i \in I_0} \xi_i a_i, \quad Y = \sum_{i \in [n] \setminus I_0} \xi_i a_i,
        \]
        and let $X'$ be an independent copy of $X$.

        Fix any $x \in \CC^k$. Then by the decoupling lemma (\cref{decoupling-lemma}),
        \begin{align*}
        \prob{X + Y \in S - x} &\le \prob{X + Y \in S - x \text{ and } X' + Y \in S - x}^{1/2} \\
        &= \prob{Y \in (S - x - X) \cap (S - x - X')}^{1/2}
        \end{align*}
        Define a (random) variety $T = (S - x - X) \cap (S - x - X')$.

        First we deal with the case when $\dim T \le \ell - 1$. By B\'ezout's theorem (\cref{Bezout}), we have $\deg T \le d^2$. Applying the induction hypothesis to the variety $T$ and the subsequences $A_1, \ldots, A_{\ell}$, we obtain that
        \begin{align} \label{eq:dim T < ell}
        \probs{Y}{Y \in T \mid \dim T \le \ell - 1} 
        &\le \ell d^2\left(\sup_{i', V'} \rho(A_{i'}, V')\right)^{1/2^{\ell-1}}.
        \end{align}
        where the supremum is taken over $1 \le i' \le \ell$ and over linear subspaces $V'$ contained in a translate of $T$. Since $T$ itself is contained in a translate of $S$, this supremum is bounded above by the supremum appearing in the statement of the lemma:
        \[
        \sup_{1 \le i' \le \ell, V' \subseteq T - y'} \rho(A_{i'}, V') \le \sup_{0 \le i \le \ell, V \subseteq S - y} \rho(A_i, V).
        \]
        
        Next, we consider the case when $\dim T = \dim S = \ell$. As $S$ is irreducible, this can happen only if $S - x - X = S - x - X'$. Define 
        \[
        V_S = \{v \in \CC^k : S = S - v\}.
        \]
        Equivalently, $V_S$ consists of all vectors $v \in \CC^k$ such that for any $x \in S$ the point $x + v$ also lies in $S$. We claim that $V_S$ is a linear subspace.
        
        It is clear that if $v_1, v_2$ lie in $V_S$ then $v_1 + v_2$ also lies in $V_S$. Thus, for any $y \in S$ and $v \in V_S$ the point $y + tv$ lies in $S$ for any $t \in \NN$. So, the variety $S$ has infinitely many intersection points with the line $y + tv$. Then it contains the whole line, and $y + tv \in S$ for any $t \in \CC$. Hence, if $v \in V_S$ then $tv \in V_S$ for any $t \in \CC$.

        Therefore, $V_S$ is a linear subspace contained in $S-y$ for any $y \in S$. From its definition, we have 
        \begin{equation} \label{eq:dim T = ell}
        \probs{X, X'}{\dim T = \ell} = \probs{X, X'}{(X + x) - (X' + x) \in V_S} = \EE_{X'}\Big[\probs{X}{X \in V_S + X' \mid X'}\Big]\le \rho(A_0, V_S).
        \end{equation}

        Combining \cref{eq:dim T < ell} and \cref{eq:dim T = ell}, we conclude that
        \begin{align*}
        \prob{Y \in T} &\le \probs{X, X'}{\dim T = \ell} + \EE_{X, X'}\Big[\probs{Y}{Y \in T \mid \dim T \le \ell - 1}\Big] \\
        &\le \rho(A_0, V_S) + \ell d^2 \left(\sup_{1 \le i' \le \ell, V' \subseteq T - y'} \rho(A_{i'}, V')\right)^{1/2^{\ell-1}} \\
        &\le (\ell + 1) d^2 \left(\sup_{0 \le i \le \ell, V \subseteq S - y} \rho(A_i, V)\right)^{1/2^{\ell - 1}}.
        \end{align*}
        So, for any $x \in \CC^k$ we have
        \[
        \prob{X + Y \in S - x} \le \prob{Y \in T}^{1/2} \le (\ell + 1) d \left(\sup_{i, V} \rho(A_i, V)\right)^{1/2^{\ell}}.
        \]
        As $\rho(A, S) = \sup_{x \in \CC^k} \prob{X + Y \in S - x}$, this completes the proof.
    \end{proof}

    Now we state and prove \cref{one-step}, which constitutes one step of the iterative procedure in the proof of \cref{structure-randomness}.

    \begin{proposition} \label{one-step}
        Fix $\delta, C, C_1 > 0$. Let $S \subseteq \CC^k$ be a variety of degree at most $d$. Consider a sequence $A$ of \textbf{at most} $n$ vectors in $\CC^k$ with basis packing number at least $\delta n$. Let $U$ be a subspace of $\CC^k$ satisfying $\rho(A, U) \ge n^{-C}$. Fix a decomposition of $\CC^k$ as $U \oplus W$ for some linear subspace $W$, and let $\pi_W:\CC^k \to W$ be the projection map.        
        Then at least one of the following holds:
        \begin{enumerate}
            \item[(a)] There exists a subsequence $A'$ of $A$ with basis packing number at least $(\delta/2) n$, and a variety $S' \subseteq W$ of degree at most $d^k$, such that $\pi_W^{-1}(S') \subseteq S$ and 
            \[
            \rho(A', S \setminus \pi_W^{-1}(S')) \le n^{-C_1}.
            \]
            \item[(b)] There exists a subsequence $A''$ of $A$ with basis packing number at least $(\delta/(2k)) n$, and a linear subspace $U' \subsetneq U$, such that
            \[
            \rho(A'', U') \ge n^{-C''}
            \]
            for some constant $C''$ not depending on $n$ (but possibly depending on $\delta, C, C_1, d, k$).
        \end{enumerate}
    \end{proposition}
    \begin{proof}
        First, we define a variety $S' \subseteq W$ in the following way:
        \[
        S' = \{w \in W : u + w \in S \text{ for every } u \in U\} 
        \]
        In other words, $S' = \bigcap_{u \in U} (S - u) \cap W$. By B\'ezout's theorem (\cref{Bezout}), the degree of each $(S - u) \cap W$ is at most $d$. Therefore, by \cref{infinite-intersection}, the set $S'$ is indeed a variety of degree at most $d^k$.

        Next we define a subsequence $A'$. Let $\pi_W: \CC^k \to W$ be the projection map. Then by \cref{rho(A S)-3},
        \[
        \rho(\pi_W(A)) = \rho(A, U) \ge n^{-C}.
        \]
        The basis packing number of $A$ is at least $\delta n$, thus, in particular, it contains at least $\delta n$ vectors. Since $\rho(\pi_W(A)) \ge n^{-C} \ge (\delta n)^{-C'}$ for some constant $C' = C'(\delta, C)$, 
        we can apply the optimal inverse theorem (\cref{inverse}) to the sequence $\pi_W(A)$, with $s = \delta n / 2$. As a result, we obtain a proper symmetric GAP $Q \subseteq W$ of rank $r = O_{\delta, C}(1)$, containing all but at most $\delta n / 2$ elements of $\pi_W(A)$, with
        \[
        |Q| \le K_{\delta, C} \cdot \rho(\pi_W(A))^{-1} n^{-r/2} 
        \]
        for some constant $K_{\delta, C}$ depending only on $\delta$ and $C$.

        We define $A'$ to be the subsequence of vectors $a$ in $A$ whose projections $\pi_W(a)$ lie in $Q$. Since $A'$ is obtained by removing at most $\delta n/2$ elements from $A$, its basis packing number is at least $\delta n/2$.

        Let $a'_1, \ldots, a'_m$ be the vectors of $A'$, and let $X = \xi_1 a'_1 + \ldots + \xi_m a'_m$. Suppose that condition (a) does not hold: that is, for some $x \in \CC^k$ 
        \begin{equation} \label{eq:condition-(1)}
        \prob{X + x \in S \setminus \pi_W^{-1}(S')} > n^{-C_1}.
        \end{equation}
        
        The projection of each vector in the sequence $A'$ onto $W$ lies in $Q$, thus for any outcomes of independent Rademacher random variables $\xi_1, \ldots, \xi_m$, the sum $\xi_1 \pi_W(a'_1) + \ldots + \xi_m \pi_W(a'_m)$ lies in the dilated GAP $n Q$. (In fact, with high probability it lies in the smaller dilated GAP $\sqrt{n \log n} Q$ by \cref{Chernoff}, but here this is not important for us.)  In particular, the random variable $\pi_W(X + x)$ is supported on some finite set $H \subseteq W$ satisfying
        \begin{equation}\label{eq:|H|}
        |H| \le n^r|Q| \le n^r \cdot K_{\delta, C} \cdot \rho(\pi_W(A))^{-1} n^{-r/2} \le K_{\delta, C} \cdot n^{C + r/2}.
        \end{equation}

        Choose $C''$ such that $K_{\delta, C} \cdot k \cdot d \cdot n^{-{C''}/{2^{k-1}} + C + {r}/{2} + C_1} < 1$ for any $n \ge 2$. Suppose that condition (b) also does not hold: that is, any subsequence $A''$ of $A$ with basis packing number at least $(\delta/(2k)) n$ and any proper subspace $U' \subsetneq U$ satisfy
        \[
        \rho(A'', U') < n^{-C''}.
        \]
        Our goal is to show that this leads to a contradiction.

        For any $w \in W$ let $S_w$ be the intersection of $S$ with the affine subspace $U + w$. Then 
        \[
        S \setminus \pi_W^{-1}(S') = \bigsqcup_{w \in W, w \notin S'} S_w.
        \]
        Moreover, the probability that $X + x$ lies in $S_w$ is positive only if $w =\pi_W(S_w)$ lies in $H$. Therefore, 
        \begin{equation} \label{eq:X + x in S}
        \prob{X + x \in S \setminus \pi_W^{-1}(S')} = \sum_{w \in H, w \notin S'} \prob{X + x \in S_w} \le \sum_{w \in H, w \notin S'} \rho(A', S_w).
        \end{equation}
        In order to estimate each summand we use \cref{iterated-decoupling}. As the basis packing number of the sequence $A'$ is at least $\delta n / 2$, we can partition it into $k$ subsequences $A_0, \ldots, A_{k-1}$, such that each of them has basis packing number at least $\delta n/(2k)$.
        
        Fix $w \in H \setminus S'$. Then, by definition of $S'$, we have $S_w \subsetneq U + w$. Note that $\dim S_w \le k-1$ and (by B\'ezout's theorem (\cref{Bezout})) $\deg S_w \le \deg S \le d$. Therefore, by \cref{iterated-decoupling},
        \[
        \rho(A', S_w) \le k \cdot d \cdot \left(\sup_{0 \le i \le k-1, \; V \subseteq S_w - y} \rho(A_i, V)\right)^{1/2^{k-1}}.
        \] 
        Crucially, since $S_w \subsetneq U + w$, the supremum in the right hand side is taken only over proper subspaces of $U$. As we assumed that condition (b) does not hold, we conclude that
        \[
        \rho(A', S_w) \le k \cdot d \cdot n^{-C'' / 2^{k-1}}.
        \]
        Combining this with \cref{eq:|H|} and \cref{eq:X + x in S}, we obtain
        \[
        \prob{X + x \in S \setminus \pi_W^{-1}(S')} \le |H| \cdot k \cdot d \cdot n^{-C'' / 2^{k-1}} \le K_{\delta, C} \cdot k \cdot d \cdot n^{-{C''}/{2^{k-1}} + C + {r}/{2}}.
        \]
        This is less than $n^{-C_1}$ by our choice of $C''$, which contradicts our assumption \cref{eq:condition-(1)}.
    \end{proof}

    Now we show how to iterate \cref{one-step} to prove \cref{structure-randomness}.

    \begin{proof}[\textbf{Proof of \cref{structure-randomness}}]
        We will describe an iterative process to construct a descending chain $A_0, A_1, \ldots $ of subsequences of $A$, and a descending chain $U_0 \supsetneq U_1 \supsetneq \ldots $ of linear subspaces of $\CC^k$, maintaining the following two properties:
        \begin{itemize}
            \item The basis packing number of $A_i$ is at least $(\delta/(2k)^i) n$;
            \item $\rho(A_i, U_i) \ge n^{-C'_i}$ for some constant $C'_i$ not depending on $n$.
        \end{itemize}
        
        We start with $A_0 = A$, $U_0 = \CC^k$ and $C'_0 = 1$. Now, suppose we have already constructed $A_i$ and $U_i$. We will attempt to find a decomposition $U\oplus W$ with $U=U_i$ (and some $A',S'$), satisfying the desired properties (1), (2) and (3) (actually, only (3) is nontrivial). If this is not possible, we will show how to construct $A_{i+1}$ and $U_{i+1}$, to continue the process. Since this process can continue for at most $k$ steps (as it is not possible to have a descending chain of more than $k+1$ linear subspaces of $\mb C^k$), this is sufficient to prove \cref{structure-randomness}.

        Let $W_i$ be an arbitrary linear complement of $U_i$, and apply \cref{one-step} to the sequence $A_i$ and the decomposition $\CC^k = U_i \oplus W_i$. Suppose that condition (a) of \cref{one-step} holds. Then there exists a variety $S' \subseteq W_i$ of degree at most $d^k$, and a subsequence $A'$ of $A_i$ with basis packing number at least $(\delta/(2(2k)^i)) n$ such that
        \[
        \rho(A', S \setminus \pi_{W_i}^{-1}(S')) \le n^{-C_1}.
        \]
        In this case we are done by setting $U = U_i$, $W = W_i$, and $C = C'_i$. Indeed, properties (1) and (3) are satisfied by the above. Recalling \cref{rho(A S)-1,rho(A S)-3}, we note that $\rho(A', U_i) \ge \rho(A_i, U_i)=\rho(\pi_{W_i}(A_i)) \ge n^{-C'_i}$, thus property (2) holds as well.
                
        Otherwise, condition (b) of \cref{one-step} holds, which gives us a subsequence $A_{i+1}$ with basis packing number at least $(\delta/(2k)^{i+1}) n$, and a subspace $U_{i+1} \subsetneq U_i$ such that 
        \[
        \rho(A_{i+1}, U_{i+1}) \ge n^{-C'_{i+1}}
        \]
        for some $C'_{i+1}$ depending only on $\delta, C'_{i}, C_1, d, k$. This allows us to proceed to the next step of the process.
    \end{proof}

    \section{Proofs of the main results}
    \label{sec:main-proofs}

    In this section we prove a general result (\cref{varieties-general} below), which allows us to reduce Littlewood--Offord-type statements about estimating $\rho(A, S)$ for an algebraic variety $S$ to questions about counting lattice points on $S$ (in the sense of the integer point density function from \cref{def:density}). We will then show how to apply \cref{varieties-general} along with the number-theoretic results presented in \cref{sec:number-theory} to derive \cref{varieties-1/d,varieties-(k-ell)/2,polynomials-all}. After that, we deduce \cref{polynomials-1/d,semialgebraic} from \cref{varieties-1/d}.
    
    \begin{theorem} \label{varieties-general}
        Let $S \subseteq \CC^k$ be a variety of dimension at most $\ell$ and degree at most $d$. Consider a sequence $A$ of vectors in $\CC^k$ with basis packing number at least $b \ge 2$. Then there exists $r = O_{d, k}(1)$ such that
        \[
        \rho(A, S) \le O_{d, k}\left(\left(d_{S}(\sqrt{b \log b}) + (b \log b)^{-(k-\ell+1)/2}\right) \cdot (\log b)^r \right).
        \]
    \end{theorem}

    In comparison to \cref{lattice-points-reduction}, note that \cref{varieties-general} does not require that $\rho(A) \ge n^{-C}$, but instead requires that $S$ is an algebraic variety of bounded degree (and it gives a slightly worse bound).
    
    \begin{proof}[\textbf{Proof of \cref{varieties-general}}]
        Let $S_1, \ldots, S_t$ (where $t \le d$) be the irreducible components of $S$. By \cref{rho(A S)-2}, we have $\rho(A, S) \le \sum_{j = 1}^t \rho(A, S_j)$. Therefore, by treating each irreducible component separately, we may assume that $S$ is irreducible.

        We have a sequence $A$ of vectors in $\CC^k$ with basis packing number at least $b$. Consider a subsequence $A_0$ of $A$, which contains only the vectors of the $b$ bases. It has size $m = bk$ and basis packing number equal to $b$. 
        Applying \cref{structure-randomness} to the sequence $A_0$ and the variety $S$ with $\delta = 1/k$ and $C_1 = k+1$, we obtain a decomposition $\CC^k = U \oplus W$, a variety $S' \subseteq W$ of degree at most $d^k$, and a subsequence $A'$ satisfying the following conditions: 
        \begin{enumerate}
            \item[(1)] The basis packing number of $A'$ is at least $\a m$ for some $\a = \a(k) > 0$;
            \item[(2)] Let $\pi_W:\CC^k \to W$ be the projection map. Then for some $C = C(d, k)$ we have $\rho(\pi_W(A')) \ge m^{-C}$;
            \item[(3)] $\pi_W^{-1}(S') \subseteq S$, and $\rho(A', S \setminus \pi_W^{-1}(S')) \le m^{-(k+1)}$.
        \end{enumerate}
        Assuming that $b$ is sufficiently large with respect to $d$ and $k$, condition (3) guarantees that  
        \begin{equation} \label{eq:rho(A' S setminus pi_W^{-1}(S'))}
        \rho(A', S \setminus \pi_W^{-1}(S')) \le m^{-(k+1)} \le b^{-(k+1)} \le (b \log b)^{-(k-\ell+1)/2} (\log b)^r.
        \end{equation}        
        By \cref{rho(A S)-1,rho(A S)-2},
        \[
        \rho(A, S) \le \rho(A', S) \le \rho(A', \pi_W^{-1}(S')) + \rho(A', S \setminus \pi_W^{-1}(S')).
        \]
        Then, by \cref{eq:rho(A' S setminus pi_W^{-1}(S'))}, the second summand $\rho(A', S \setminus \pi_W^{-1}(S'))$ is small compared to the desired upper bound. Therefore, it is sufficient to focus on the first summand $\rho(A', \pi_W^{-1}(S'))$. \cref{rho(A S)-3} implies that
        \[
        \rho(A', \pi_W^{-1}(S')) = \rho(\pi_W(A'), S').
        \]
        
        By condition (1), the basis packing number of the sequence $A'$ is at least $\a m$, thus the same holds for $\pi_W(A')$ (as a sequence of vectors in $W$). Let $m'$ be the size of $A'$, so we have $\a m \le m' \le m$. Condition (2) then implies that $\rho(\pi_W(A')) \ge m^{-C} \ge (m')^{-C'}$ for some constant $C' = C'(C, k)$. Therefore, we can apply \cref{lattice-points-reduction} to the sequence $\pi_W(A')$ and the variety $S'$ with $\FF = \CC$, $\delta = \a$ and $C_1 = k+1$. As a result, we obtain a positive integer $r = O_{d, k}(1)$ such that
        \[
            \rho(\pi_W(A'), S') \le O_{d, k}\left(d_{S'}(\sqrt{m' \log m'}) \cdot (\log m')^r + (m')^{-(k+1)}\right).
        \]     
        Since $\a b \le \a m \le m' \le m = bk$ and the density function does not ``jump too much'' by \cref{no-jumps}, we conclude that
        \begin{equation} \label{eq:d_{S'}(sqrt{m log m})}
        \rho(\pi_W(A'), S') \le O_{d, k}\left(d_{S'}(\sqrt{b \log b}) \cdot (\log b)^r + b^{-(k+1)}\right).
        \end{equation}        
        Again, $b^{-(k+1)}$ is small compared to the desired upper bound. The expression in the first summand is quite similar to the one in the statement of the theorem, except that it involves the density function of $S'$ instead of $S$. We consider two cases depending on whether $\pi_W^{-1}(S') = S$ or not.
        
        First, suppose that $\pi_W^{-1}(S') = S$. Then, by \cref{slicing}, we have $d_{S'} = d_S$. Substituting this into \cref{eq:d_{S'}(sqrt{m log m})} gives the desired bound, completing the proof in this case.
        
        Otherwise, we have $\pi_W^{-1}(S') \subsetneq S$. As $S$ is irreducible, we combine this with \cref{projections} to conclude that
        \[
        \codim S' = \codim(\pi_W^{-1}(S')) \ge \codim S + 1 \ge k-\ell+1.
        \]
        Recall that degree of $S'$ is at most $d^k$. Then from the Schwartz--Zippel lemma (\cref{trivial-bound}) we have
        \[
        d_{S'}(\sqrt{b \log b}) \le O_{d, k}\Big((b \log b)^{-(k-\ell+1)/2}\Big).
        \]
        Again, substituting this into \cref{eq:d_{S'}(sqrt{m log m})} completes the proof.
    \end{proof}

    \begin{proof}[\textbf{Proof of \cref{varieties-1/d}}]
        The probability $\prob{\xi_1 a_1 + \ldots + \xi_n a_n \in S}$ is bounded by $\rho(A, S)$, and we would like to prove that
        \[
        \rho(A, S) \le O_{d, k}\big(b^{-(k-\ell+1-1/d)} (\log b)^{C_{d, k}}\big).
        \]
        By \cref{varieties-general}, for some $r = O_{d, k}(1)$ we have
        \begin{equation} \label{eq:d_S(sqrt{b log b})}
        \rho(A, S) \le O_{d, k}\left(\left(d_{S}(\sqrt{b \log b}) + (b \log b)^{-(k-\ell+1)/2}\right) \cdot (\log b)^r\right),
        \end{equation}
        Recall that $S \subseteq \CC^k$ is an irreducible variety of dimension $\ell$ and degree $d$. Then, by Pila's bound (\cref{Pila}),
        \[
        d_{S}(\sqrt{b \log b}) = \sup_{\varphi}\left(\frac{N_{\varphi(S)}(\sqrt{b \log b})}{(2 \lfloor\sqrt{b \log b}\rfloor + 1)^k}\right) \le O_{d, k}\left(b^{-(k-\ell+1-1/d)/2} (\log b)^{C_{d}-(k-\ell+1-1/d)/2}\right).
        \]
        Substituting this into \cref{eq:d_S(sqrt{b log b})} gives $\rho(A, S) \le O_k\left(b^{-(k-l+1-1/d)/2} (\log b)^{C_d+r-(k-l+1-1/d)/2}\right)$, completing the proof.
    \end{proof}

    Next, we deduce \cref{varieties-(k-ell)/2} by combining \cref{varieties-1/d} with the following result of Ferber, Jain and Zhao \cite{ferber-jain-zhao-22} (which is a refined version of Hal\'asz' theorem \cite{halasz-77}). 

    \begin{theorem}[{\cite[Theorem 1.11]{ferber-jain-zhao-22}}] \label{halasz}
        Let $A$ be a sequence of vectors in $\RR^k$, and let $I_1, \ldots, I_{s}$ (for some even $s$) be a partition of $[n]$. Denote $t := \frac{1}{s}\sum_{j = 1}^s \dim_\RR \Span_{\RR}\{a_i : i \in I_j\}$. Then
        \[
        \sup_{x \in \RR^k} \prob{\xi_1 a_1 + \ldots + \xi_n a_n = x} \le \left(2^{-s}\binom{s}{s/2}\right)^{t}. 
        \]
    \end{theorem}

    \begin{proof}[\textbf{Proof of \cref{varieties-(k-ell)/2}}]
        Let $S_1$ be an irreducible component of $S$. If $\deg S_1 \ge 2$, we apply \cref{varieties-1/d} to conclude that
        \[
        \prob{\xi_1 a_1 + \ldots + \xi_n a_n \in S_1} \le O_{d, k}\left(b^{-(k-\dim S_1 +1-1/\deg S_1)/2} (\log b)^{C}\right).
        \]
        for some constant $C$ depending only on $\deg S_1$ and $k$. Since 
        \[
        \dim S_1 \le \dim S \le \ell \quad \text{ and } \quad 2 \le \deg S_1 \le \deg S \le d,
        \] 
        this bound is at most $O_{d, k}(b^{-(k-\ell)/2})$. This completes the proof in this case.

        If $\deg S_1 = 1$ then $S_1$ is an affine subspace of dimension at most $\ell$ (by enlarging it, we may assume that its dimension is exactly $\ell$). A minor technical issue we need to handle is that \cref{halasz} is stated only for point probabilities and only over the field $\RR$. 
        
        Let $V$ be the linear subspace which is a translate of $S_1$, and let $\pi:\CC^k \to \CC^k/V \simeq \CC^{k-\ell} \simeq \RR^{2(k-\ell)}$ be the quotient map. By assumption, there is a  partition $I_1, \ldots, I_{b_1}$ of $[n]$ with $b_1 = 2\lfloor b/2 \rfloor \le b$, such that each subset of vectors $\{a_i : i \in I_j\}$ contains a basis of $\CC^k$. Then, for any $1 \le j \le b_1$,
        \[
        \dim_{\RR}\Span_{\RR}\{\pi(a_i) : i \in I_j\} \ge \dim_{\CC}\Span_{\CC}\{\pi(a_i) : i \in I_j\} = k-\ell.
        \]
        Applying \cref{halasz} to this partition, we conclude that
        \[
        \prob{\xi_1 a_1 + \ldots + \xi_n a_n \in S_1} \le \sup_{x \in \CC^k/V} \prob{\xi_1 \pi(a_1) + \ldots + \xi_n \pi(a_n) = x} \le \left(2^{-b_1}\binom{b_1}{b_1/2}\right)^{k-\ell} \le O(b^{-(k-\ell)/2}).
        \]
        Summing over all irreducible components of $S$ (there are at most $d$ of them) completes the proof.
    \end{proof}

    Next, we deduce \cref{polynomials-all,polynomials-1/d,semialgebraic} from the results proved previously in this section. All these deductions share the same first step, which is an application of the ``dropping to a subspace'' lemma (\cref{dropping-to-subspace}).
    
    \begin{proof}[\textbf{Proof of \cref{polynomials-all}}]
        As $F$ has Chow rank at most $c$, it can be written as 
        \[
        F(t_1, \ldots, t_n) = f(L_1(t_1, \ldots, t_n), \ldots, L_k(t_1, \ldots, t_n))
        \]        
        for $k = dc$, some $f \in \CC[x_1, \ldots, x_k]$ and homogeneous linear forms $L_1, \ldots, L_k$. Suppose that the form $L_i$ is given by $a_{i1} t_1 + \ldots + a_{in} t_n$ for some coefficients $a_{ij} \in \CC$. Denoting $a_j = (a_{1j}, \ldots, a_{kj}) \in \CC^k$, we have
        \[
        F(t_1, \ldots, t_n) = f(t_1 a_1 + \ldots + t_n a_n).
        \]
        Let $b_0 = \lfloor b/(k(k+1)) \rfloor + 1$. Then by \cref{dropping-to-subspace} applied to the sequence $A = (a_1, \ldots, a_n)$, there exists a subsequence $A' = A[I_0]$ of size at least
        \[
        n - (b_0-1)\frac{k(k+1)}{2} \ge n - b/2,
        \]
        and a subspace $V' \subseteq \CC^k$ (of dimension $k' \le k$) such that all the elements of $A'$ lie in $V'$, and $A'$ has basis packing number at least $b_0$ (as a sequence of vectors in $V'$). 

        Let $I_1 = [n]\setminus I_0$, so $|I_1| \le b/2 < b$. It suffices to show that for an arbitrary outcome of the Rademacher random variables $(\xi_i)_{i \in I_1}$, if we condition on $(\xi_i)_{i \in I_1}$ taking this particular outcome, then the desired bounds on $\prob{F(\xi_1,\dots,\xi_n)=0}$ hold in the resulting conditional probability space. In other words, let $F_*$ be a polynomial obtained by an arbitrary substitution of $\pm 1$ instead of the variables $(t_i)_{i \in I_1}$; it suffices to prove the desired bounds with ``$F_*$'' in place of ``$F$''.
        
        Note that we can write
        \begin{equation} \label{eq:poly-all-0}
        F_*((t_i)_{i \in I_0}) = f_*\left(\sum_{i \in I_0} t_i a_i\right)
        \end{equation}
        for some polynomial $f_*$ defined on $V'$ (one can take $f_*(x)$ to be a restriction of $f(x + x_0)$ to $V'$, for certain $x_0 \in \CC^k$). 
        Therefore, we need to estimate the probability that $f_*\left(\sum_{i \in I_0} \xi_i a_i\right) = 0$.

        Let $S \subseteq V' \simeq \CC^{k'}$ be the variety defined by $f_*$.
        \begin{enumerate}
            \item [(1)] First, we prove \cref{polynomials-all}(1). By our assumption, $F_*$ is not identically zero. Since $b_0 \ge 1$, the vectors $(a_i)_{i \in I_0}$ span the vector space $V'$, and thus the polynomial $f_*$ is also not identically zero. By \cref{one-polynomial} (applied to each irreducible factor of $f_*$ over $\CC$) combined with \cref{unions}, $S$ has dimension $k'-1$ and degree at most $\deg f_* \le \deg f = d$.

            As the basis packing number of the sequence $A'$ is at least $b_0$, from \cref{varieties-(k-ell)/2} we obtain that
            \[
            \prob{f_*\left(\sum_{i \in I_0} \xi_i a_i\right) = 0} = \prob{\sum_{i \in I_0} \xi_i a_i \in S} \le O_{d, k}(b_0^{-1/2}).
            \]
            Since $b_0 = \Omega_k(b)$, this completes the proof.
            
            \item[(2)] Next, we prove \cref{polynomials-all}(2). We may assume that $b > 2d$ (otherwise, the desired probability bound is trivial). Let $F^{=d}_{*}$ be the homogeneous degree-$d$ part of $F_{*}$. The polynomial $F_*$, in turn, was obtained from $F$ by substitution of $\pm1$'s instead of at most $b/2$ of its variables. Observe that each variable of $F$ is part of at most $n^{d-1}$ monomials of degree $d$, and that there are at most $dn^{d-1}$ monomials of degree less than $d$ in total. Therefore, $F^{=d}_{*}-F$ has at most $(b/2+d)n^{d-1} < bn^{d-1}$ nonzero coefficients.
            So, by our assumption on $F$, the polynomial $F^{=d}_{*}$ is irreducible (in particular, it is not zero). Recalling $\cref{eq:poly-all-0}$, we conclude that the homogeneous degree-$d$ part $f^{=d}_*$ of $f_*$ is also irreducible and nonzero.
            
            So, $f_*$ is an irreducible polynomial of degree $d$. Moreover, we note that $f_*$ cannot be written as a polynomial of two linear forms. Indeed, otherwise $f_*^{=d}$ can be represented as $g(L_1, L_2)$ for a homogeneous polynomial $g \in \CC[x_1, x_2]$ and two homogeneous linear forms $L_1, L_2$. Recall that $\CC$ is algebraically closed, and let $r_1, \ldots, r_{d_*}$ be the complex roots of $g(1, z)$. Then
            \[
            g(L_1, L_2) = L_1^d g(1, L_2/L_1) = z_0 L_1^d \prod_{j = 1}^{d_*}(L_2/L_1 - r_j) = z_0 L_1^{d-d_*} \prod_{j = 1}^{d_*}(L_2 - r_j L_1) \quad \text{ (for some $z_0 \in \CC$)}.
            \]
            So, $f_*^{=d}$ splits into a product of linear forms. As $d \ge 1$, this contradicts its irreducibility.
            
            Since the variety $S \subseteq \CC^{k'}$ is defined by the irreducible polynomial $f_*$ of degree $d$, it has dimension $k'-1$ and degree $d$ by \cref{one-polynomial}. Recall that the basis packing number of the sequence $A'$ is at least $b_0$. Applying \cref{varieties-general} to $A'$ and $S$, we obtain that
            \begin{equation} \label{eq:poly-all-1}
                \prob{f_*\left(\sum_{i \in I_0} \xi_i a_i\right) = 0} = \prob{\sum_{i \in I_0} \xi_i a_i \in S} \le O_{d, k}\left(\left(d_{S}(\sqrt{b_0 \log b_0}) + (b_0 \log b_0)^{-1}\right) \cdot (\log b_0)^r \right)
            \end{equation}
            for some $r = O_{d, k}(1)$. 
            Since $f_*$ is an irreducible polynomial of degree $d \neq 3$ which cannot be represented as a polynomial of two linear forms, we can estimate the density function $d_S$ using the result of Vermeulen and Browning--Gorodnik (\cref{affine-dimension-growth}). Namely, we conclude that for any $\eps > 0$
            \[
            d_S\left(\sqrt{b_0 \log b_0}\right) \le O_{d, k, \eps}\left(b_0^{-1+\eps}\right).
            \]
            Since $b_0 = \Omega_k(b)$, substituting this into \cref{eq:poly-all-1} completes the proof. \qedhere
        \end{enumerate}
    \end{proof}
    \begin{remark}
        In fact, the proof of \cref{polynomials-all}(2) implies the following slightly stronger statement. Let $F \in \CC[t_1, \ldots, t_n]$ be a polynomial of degree $d \neq 3$ and Chow rank at most $c$. Suppose that after any substitution of $\pm 1$'s instead of fewer than $b$ variables of $F$, the resulting polynomial is irreducible of degree $d$ and cannot be written as a polynomial of two linear forms. Then for any $\eps > 0$ we have $\prob{F(\xi_1, \ldots, \xi_n) = 0} = O_{d, k, \eps}(b^{-1+\eps})$.
    \end{remark}

    \begin{proof}[\textbf{Proof of \cref{polynomials-1/d}}]
        The first half of the proof is almost identical to the first half of the proof of \cref{polynomials-all}. As $F$ has Chow rank at most $c$ (over $\FF$), it can be written as 
        \[
        F(t_1, \ldots, t_n) = f(L_1(t_1, \ldots, t_n), \ldots, L_k(t_1, \ldots, t_n))
        \]        
        for $k = dc$, some $f \in \FF[x_1, \ldots, x_k]$ and homogeneous linear forms $L_1, \ldots, L_k$ with coefficients in $\FF$. Suppose that the form $L_i$ is given by $a_{i1} t_1 + \ldots + a_{in} t_n$ with $a_{ij} \in \FF$. Denoting $a_j = (a_{1j}, \ldots, a_{kj}) \in \FF^k$, we have
        \[
        F(t_1, \ldots, t_n) = f(t_1 a_1 + \ldots + t_n a_n).
        \]
        Let $b_0 = \lfloor b/(k(k+1)) \rfloor + 1$. Then by \cref{dropping-to-subspace} applied to the sequence $A = (a_1, \ldots, a_n)$, there exists a subsequence $A' = A[I_0]$ of size at least $n - b/2$ and a subspace $V' \subseteq \FF^k$ (of dimension $k' \le k$) such that all the elements of $A'$ lie in $V'$, and $A'$ has basis packing number at least $b_0$ (as a sequence of vectors in $V'$). 

        Let $I_1 = [n]\setminus I_0$, $|I_1| \le b/2 < b$. It suffices to show that for an arbitrary outcome of the Rademacher random variables $(\xi_i)_{i \in I_1}$, if we condition on $(\xi_i)_{i \in I_1}$ taking this particular outcome, then the desired bounds on $\prob{F(\xi_1,\dots,\xi_n)=0}$ hold in the resulting conditional probability space. In other words, let $F_*$ be a polynomial obtained by an arbitrary substitution of $\pm 1$ instead of variables $(t_i)_{i \in I_1}$; it suffices to prove the desired bounds with ``$F_*$'' in place of ``$F$''.
    
        Note that we can write
        \begin{equation} \label{eq:poly-all-0}
        F_*((t_i)_{i \in I_0}) = f_*\left(\sum_{i \in I_0} t_i a_i\right)
        \end{equation}
        for some polynomial $f_*$ with coefficients in $\FF$ defined on $V'$ (one can take $f_*(x)$ to be a restriction of $f(x + x_0)$ to $V'$, for certain $x_0 \in \FF^k$). 
        Therefore, we need to estimate the probability that $f_*\left(\sum_{i \in I_0} \xi_i a_i\right) = 0$.
        
        By our assumption, $F_*$ is irreducible (over $\FF$) of degree $d$. Since $b_0 \ge 1$, the vectors $(a_i)_{i \in I_0}$ span the vector space $V'$, and thus $f_*$ is also irreducible (over $\FF$) of degree $d$. First, we consider the case when $f_*$ is, furthermore, irreducible over $\CC$.

        As we have $V' \simeq \FF^{k'} \subseteq \CC^{k'}$, let $S \subseteq \CC^{k'}$ be the variety defined by $f_*$. By \cref{one-polynomial}, it is an irreducible variety of dimension $k'-1$ and degree $d$. As the basis packing number of $A'$ (as a sequence of vectors in $V'$) is at least $b_0$, we apply \cref{varieties-1/d} to conclude that
        \[
        \prob{f_*\left(\sum_{i \in I_0} \xi_i a_i\right) = 0} = \prob{\sum_{i \in I_0} \xi_i a_i \in S} \le O_{d, k}\left(b_0^{-1 + \frac{1}{2d}} (\log b_0)^{C_{d, k}}\right).
        \]
        Since $b_0 = \Omega_k(b)$, this completes the proof in this case.

        Next, we consider the case when $f_*$ is reducible over $\CC$. Let $g$ be any irreducible (over $\CC$) factor of $f_*$. Since $f_*$ is irreducible over $\FF$, $g$ is not proportional to a polynomial with coefficients in $\FF$. Thus, by \cref{galois-trick}, there exists a variety $T \subseteq \CC^{k'}$ of dimension at most $k'-2$ and degree at most $d^2$ such that $T \cap \FF^{k'} = S_g \cap \FF^{k'}$ (where $S_g \subseteq \CC^{k'}$ is the variety defined by $g$). Note that $\sum_{i \in I_0} \xi_i a_i$ always lies in $\FF^{k'}$. As the basis packing number of $A'$ is at least $b_0$, from \cref{varieties-(k-ell)/2} we conclude that
        \[
        \prob{g\left(\sum_{i \in I_0} \xi_i a_i\right) = 0} = \prob{\sum_{i \in I_0} \xi_i a_i \in S_g} = \prob{\sum_{i \in I_0} \xi_i a_i \in T} \le O_{d, k}(b_0^{-1}).
        \]        
        Since $b_0 = \Omega_k(b)$, this is at most $O_{d, k}(b^{-1})$. Taking the sum over all irreducible (over $\CC$) factors of $f_*$ completes the proof.
    \end{proof}

    Before proceeding to the proof of \cref{semialgebraic}, we record the following proposition about the Zariski closure of a semialgebraic set. Our main reference for properties of semialgebraic sets is the notes of Coste \cite{coste-00}.

    \begin{proposition} \label{semialgebraic-dimension}
        Let $S\subseteq \RR^k \subseteq \CC^k$ for $k \ge 1$ be a semialgebraic set which does not contain a line segment. Then the Zariski closure of $S$ in $\CC^k$ is a variety of dimension at most $k-1$ (equivalently, this Zariski closure is not the whole $\CC^k$).
    \end{proposition}
    \begin{proof}
        The statement follows from these three facts:
        \begin{itemize}
        \item \cite[Proposition 3.15]{coste-00} The dimension of $S$ as a semialgebraic set is defined as the maximum dimension of a cell in its cell decomposition. A cell of dimension $k$ is homeomorphic to $(0, 1)^k$, and therefore contains a line segment. Therefore, the dimension of $S$ as a semialgebraic set is at most $k-1$.
        \item \cite[Theorem 3.20]{coste-00} The dimension of $S$ as a semialgebraic set is equal to the dimension of its \textit{real} Zariski closure $S_{\RR}$ (as a \textit{real} algebraic set).
        \item The dimension of $S_{\RR}$ as a \textit{real} algebraic set is equal to the dimension of its \textit{complex} Zariski closure $S_{\CC}$ (as a \textit{complex} algebraic variety). Since the dimension of a variety is equal to the Krull dimension of its coordinate ring \cite[Proposition 1.7]{hartshorne-77}, this can be seen (for example) from the Noether normalization lemma. \qedhere
        \end{itemize}
    \end{proof}

    \begin{proof}[\textbf{Proof of \cref{semialgebraic}}]
        Let $b_0 = \lfloor n/(k(k+1)) \rfloor + 1$. Then by \cref{dropping-to-subspace} applied to the sequence $A = (a_1, \ldots, a_n)$, there exists a subsequence $A' = A[I_0]$ of size at least $n/2$ and a subspace $V' \subseteq \RR^k$ such that all the elements of $A'$ lie in $V'$, and $A'$ has basis packing number at least $b_0$ (as a sequence of vectors in $V'$). Conditioning on the outcomes of the random variables $(\xi_i)_{i \in [n] \setminus I_0}$, we have
        \begin{equation} \label{eq:semialgebraic-1}
        \prob{\xi_1 a_1 + \ldots + \xi_n a_n \in S} \le \sup_{x \in \RR^k} \prob{\sum_{i \in I_0} \xi_i a_i \in (S-x) \cap V'}.
        \end{equation}        
        As $V' \subseteq \RR^k \subseteq \CC^k$, define $V'_{\CC} \subseteq \CC^k$ as be the minimal complex subspace containing $V'$. It satisfies $V'_{\CC} \cap \RR^k = V'$ and $\dim_{\CC} V_{\CC} = \dim_{\RR} V'$.
        
        Fix any $x \in \RR^k$. Note that $(S-x) \cap V'$ is a semialgebraic set which does not contain a line segment (since $S$ does not contain a line segment).
        Then, by \cref{semialgebraic-dimension}, the complex Zariski closure $S_{\CC} \subseteq V'_{\CC}$ of $(S-x)\cap V'$ has dimension at most $\dim V'_{\CC}-1$.         
        Recall that the basis packing number of the sequence $A'$ is at least $b_0$. Applying \cref{varieties-(k-ell)/2} to $A'$ and $S_{\CC}$, we conclude that
        \[
        \prob{\sum_{i \in I_0} \xi_i a_i \in (S-x) \cap V'} \le \prob{\sum_{i \in I_0}\xi_i a_i \in S_{\CC}} \le O_{\deg(S_{\CC}), k}(b_0^{-1/2}).
        \]
        Since $b_0 = \Omega_k(n)$, substituting this into \cref{eq:semialgebraic-1} completes the proof.
    \end{proof}

    \begin{remark}
        \label{semialgebraic-improved}
        The bound in \cref{semialgebraic} is sharp up to a multiplicative constant factor: if $S = \{0\}$, and $a_1 = \ldots = a_{2n} \in \RR^k\setminus\{0\}$, then $\prob{\xi_1 a_1 + \ldots + \xi_{2n} a_{2n} \in S} = \binom{2n}{n}/2^{2n} = \Theta(n^{-1/2})$. However, we sketch how this bound can be improved under additional conditions on the vectors $a_1, \ldots, a_n$.
        \begin{itemize}        
        \item Suppose that the vectors $a_1, \ldots, a_n \in \RR^k$ in \cref{semialgebraic} satisfy the following condition: for some $\delta > 0$ every line passing through the origin contains fewer than $(1-\delta)n$ of them (this happens, for example, when $k \ge 2$ and one can form $\delta n$ disjoint bases from these vectors). Then one can modify the proof above to ensure that $\dim V'_{\CC} \ge 2$. Since $S$ does not contain a line segment, a simple argument based on \cref{semialgebraic-dimension} implies that no irreducible component of its Zariski closure $S_{\CC} \subseteq V'_{\CC}$ is a hyperplane in $V'_{\CC}$. Therefore, applying \cref{varieties-1/d} instead of \cref{varieties-(k-ell)/2}, we obtain a stronger bound 
        \[
        \prob{\xi_1 a_1 + \ldots + \xi_n a_n \in S} \le O_{S, \delta}(n^{-3/4} (\log n)^{C_k}).
        \]
        \item Furthermore, suppose that the vectors $a_1, \ldots, a_n \in \RR^k$ satisfy the following condition: for some $\delta > 0$ every two-dimensional linear subspace contains fewer than $(1-\delta)n$ of them. In this case one can similarly ensure that $\dim V'_\CC \ge 3$, and that each irreducible component of the Zariski closure $S_\CC \subseteq V'_\CC$ either has codimension at least two, or is not a preimage of a curve under a linear map. Applying \cref{varieties-general} combined with the best known results on the affine dimension growth conjecture (which is settled for $d \neq 3$, but has only partial results for $d = 3$) stated for varieties ``not cylindrical over a curve'' \cite[Theorem 1.2]{vermeulen-24}, we can get a bound of roughly $O_{S, \delta}(n^{-0.92})$. 
        \end{itemize}
    \end{remark}

    \section{Concluding remarks}
In this paper we have introduced a general method to study certain geometric variants of the Littlewood--Offord problem via lattice point counting, and applied this method in several different contexts. There are a number of interesting directions for future research.

First, regarding the general polynomial Littlewood--Offord problem: one of our primary motivations to consider the bounded-rank setting was that this setting already seems to incorporate many of the most important difficulties of the general polynomial Littlewood--Offord problem. Indeed, the resolution of the quadratic Littlewood--Offord problem by Kwan and Sauermann~\cite{kwan-sauermann-23} was accomplished by first solving the bounded-rank case, and then adapting and quantifying the approach for the general case.

Unfortunately, it seems challenging to adapt the techniques in this paper to general polynomials (without a bound on the Chow rank). We remark that we do not really need the Chow rank to be $O(1)$: indeed, it should be straightforward to modify our proof of \cref{polynomials-all}(1) (by using \cref{varieties-1/d} and the Erd\H{o}s--Littlewood--Offord theorem instead of \cref{varieties-(k-ell)/2}) to show that there is a slowly growing function $h(n)$ such that when the Chow rank of $F$ is at most $h(n)$ then 
\[
\prob{F(\xi_1, \ldots, \xi_n) = 0} \le O_{d}(b^{-1/2}),
\]
with the implicit constant not depending on the Chow rank. However, $h(n)$ would definitely need to grow rather slowly (e.g., it seems significant new ideas would be required to handle Chow rank as large as $n^{0.01}$).

It may also be fruitful to consider different (``weaker'') notions of rank/complexity than Chow rank. Indeed, while there is only really one sensible notion of rank for quadratic polynomials, for polynomials of higher degree there are several fundamentally different notions of rank. One natural candidate that often arises in analytic number theory is the \emph{Schmidt rank} (also called \emph{$h$-invariant} or \emph{strength}): for a polynomial $F$ of degree $d$, its Schmidt rank is the smallest integer $s$ such that $F$ can be written as $\sum_{i = 1}^s P_i$, where each $P_i$ is a product of two polynomials of degree strictly less than $d$ (for a homogeneous polynomial of a fixed degree $d$, its Schmidt rank is equivalent to the so-called \emph{partition rank} of its coefficient tensor).

Finally, another interesting direction is to consider ``small-ball concentration'' probabilities instead of ``point concentration'' probabilities. Namely, assuming that ``sufficiently many'' coefficients of $F$ have absolute value at least $1$, one can sometimes obtain similar upper bounds on $\prob{|F(\xi_1, \ldots \xi_n)| \le 1}$ (see \cite{halasz-77,meka-nguyen-vu-16, ferber-jain-zhao-22}). It seems difficult to adapt our methods to this setting. One reason for this is that the decoupling techniques used in \cref{sec:decoupling} are not well-suited for small-ball probabilities. Another reason is that we are not aware of appropriate number-theoretic results (analogous to \cref{Pila}) sufficient to finish the proof in the small-ball setting.       

    \bibliographystyle{plain}

    \bibliography{references}

\end{document}